\documentclass[12pt,etds]{amsart}

\usepackage{amsmath}
\usepackage{amssymb}
\usepackage{pb-diagram}

\newtheorem{theorem}{Theorem}[section]
\newtheorem{definition}[theorem]{Definition}
\newtheorem{lemma}[theorem]{Lemma}
\newtheorem{prop}[theorem]{Proposition}
\newtheorem{cor}[theorem]{Corollary}
\theoremstyle{remark}
\newtheorem{example}[theorem]{Example}
\newtheorem{remark}[theorem]{Remark}
\theoremstyle{question}

\def\N{{\mathbb N}}

\def\C{{\mathbb C}}
\def\R{{\mathbb R}}
\def\TT{{\mathbb T}}

\def\Z{{\mathbb Z}}
\def\A{{\mathcal{A}}}

\def\K{{\mathcal{K}}}

\def\L{{\mathcal{L}}}
\def\T{{\mathcal{T}}}
\def\I{{\mathcal{I}}}

\newcommand{\id}{\operatorname{id}}
\newcommand{\piso}{\operatorname{piso}}
\newcommand{\iso}{\operatorname{iso}}
\newcommand{\End}{\operatorname{End}}
\newcommand{\Aut}{\operatorname{Aut}}

\newcommand{\whitesquare}{\hfill $\whitesquare$\newline\vspace{0.4cm}}

\def\newspan{\operatorname{span}}

\numberwithin{equation}{section}

\begin{document}

\title[The partial-isometric crossed products as full corners]{The partial-isometric crossed products by semigroups of endomorphisms as full corners}
\author[Sriwulan Adji]{Sriwulan Adji}
\address{Current affiliation: Institute of Mathematical Sciences, University Malaya, Kuala Lumpur (Malaysia)}
\email{sriwulan.adji@yahoo.com}

\author[Saeid Zahmatkesh]{Saeid Zahmatkesh}
\address{School of Mathematical Sciences, Universiti Sains Malaysia, Penang (Malaysia)}
\email{zahmatkesh$_{-}$s@yahoo.com}

\date{\today}

\begin{abstract}
Suppose $\Gamma^{+}$ is the positive cone of a totally ordered abelian group $\Gamma$, and $(A,\Gamma^{+},\alpha)$ is a system consisting of a $C^*$-algebra $A$, an action $\alpha$ of $\Gamma^{+}$ by extendible endomorphisms of $A$.
We prove that the partial-isometric crossed product $A\times_{\alpha}^{\piso}\Gamma^{+}$ is a full corner in the subalgebra of
$\L(\ell^{2}(\Gamma^{+},A))$, and that if $\alpha$ is an action by automorphisms of $A$, then
it is the isometric-crossed product $(B_{\Gamma^{+}}\otimes A)\times^{\iso}\Gamma^{+}$, which is therefore a full corner in the usual crossed product of  system by a group of automorphisms.
We use these realizations to identify the ideal of $A\times_{\alpha}^{\piso}\Gamma^{+}$ such that the quotient is the isometric crossed product $A\times_{\alpha}^{\iso}\Gamma^{+}$.
\end{abstract}

\maketitle

\section{Introduction}
Let $\Gamma$ be a totally ordered abelian group, and  $\Gamma^{+}:=\{x\in\Gamma : x \ge 0\}$  the positive cone of $\Gamma$.
A dynamical system $(A,\Gamma^{+},\alpha)$ is a system consisting of a $C^*$-algebra $A$, an action
$\alpha:\Gamma^{+}\rightarrow \End{A}$ of $\Gamma^{+}$ by endomorphisms $\alpha_{x}$ of $A$ such that $\alpha_{0}={\rm id}_{A}$.
Since we do not require the algebra $A$ to have an identity element, we need to assume that every endomorphism $\alpha_{x}$ extends to a strictly continuous endomorphism $\overline{\alpha}_{x}$ of the multiplier algebra $M(A)$ as it is used in \cite{Adji1,Larsen}, and note that extendibility of $\alpha_{x}$ may imply ${\alpha}_{x}(1_{M(A)})\neq 1_{M(A)}$.

A partial-isometric covariant representation, the analogue of isometric covariant representation, of the system $(A,\Gamma^{+},\alpha)$ is defined in \cite{LR}  where the endomorphisms $\alpha_{s}$ are represented by partial-isometries instead of isometries.
The partial-isometric crossed product $A\times_{\alpha}^{\piso}\Gamma^{+}$ is defined in there as the Toeplitz algebra studied in \cite{F} associated to a product system of Hilbert bimodules arises from the underlying dynamical system $(A,\Gamma^{+},\alpha)$.
This algebra is universal for covariant partial-isometric representations of the system.

The success of the theory of isometric crossed products \cite{Adji2,Adji,ALNR,murphy,murphy2,murphy3}
has led authors in \cite{LR} to study the structure of partial-isometric crossed product of the distinguished system $(B_{\Gamma^{+}},\Gamma^{+},\tau)$, where $\tau_{x}$  acts on the subalgebra $B_{\Gamma^{+}}$ of $\ell^{\infty}(\Gamma^{+})$ as the right translation.
However the analogous view of isometric crossed products as full corners in crossed products by groups \cite{Adji1,Laca,stacey} for partial-isometric crossed products remains unavailable.
This is the main task undertaken in the present work.

We construct a covariant partial-isometric representation of $(A,\Gamma^{+},\alpha)$ in the $C^{*}$-algebra $\L(\ell^{2}(\Gamma^{+},A))$ of adjointable operators on the Hilbert $A$-module $\ell^{2}(\Gamma^{+},A)$, and we show
the corresponding representation of the crossed product is an isomorphism of $A\times_{\alpha}^{\piso}\Gamma^{+}$ onto a full corner in the subalgebra of $\L(\ell^{2}(\Gamma^{+},A))$.
We use the idea from \cite{KS} for the construction: the embedding $\pi_{\alpha}$ of $A$ into $\L(\ell^{2}(\Gamma^{+},A))$, together with the isometric representation $S:\Gamma^{+}\rightarrow \L(\ell^{2}(\Gamma^{+},A))$, satisfy the equation $\pi_{\alpha}(a)S_{x}=S_{x}\pi(\alpha_{x}(a))$ for all $a\in A$ and $x\in\Gamma^{+}$, and then the algebra $\T_{(A,\Gamma^{+},\alpha)}$ generated by $\pi(A)$ and $S(\Gamma^{+})$ contains $A\times_{\alpha}^{\piso}\Gamma^{+}$ as a full corner.
However since the results in \cite{KS} are developed to compute and to show that $KK$-groups of $\T_{(A,\Gamma^{+},\alpha)}$ and $A$ are equivalent, the theory is set for unital $C^*$-algebras and unital endomorphisms: if the algebra is not unital, they use the smallest unitization algebra $\tilde{A}$ and then the extension of endomorphism on $\tilde{A}$ is unital.

Here we use the (largest unitization) multiplier algebra $M(A)$ of $A$, and every endomorphism is extendible to $M(A)$.
So we generalize the arguments in \cite{KS} to the context of multiplier algebra.
When endomorphisms in a given system are unital, then we are in the context of \cite{KS} which therefore the $C^*$-algebra $A\times_{\alpha}^{\piso}\Gamma^{+}$ enjoys all properties of the algebra $\T_{(A,\Gamma^{+},\alpha)}$ described in \cite{KS}.
Moreover if the action is automorphic action then we show that $A\times_{\alpha}^{\piso}\Gamma^{+}$ is a full corner in the crossed product by group action.

Using the corner realization of $A\times_{\alpha}^{\piso}\Gamma^{+}$, we identify the kernel  of the natural surjective homomorphism
$i_{A}\times i_{\Gamma^{+}}: A\times_{\alpha}^{\piso}\Gamma^{+} \rightarrow A\times_{\alpha}^{\iso}\Gamma^{+}$ induced by the canonical isometric covariant pair $(i_{A},i_{\Gamma^{+}})$ of $(A,\Gamma^{+},\alpha)$, and to get the exact sequence of \cite{KS} and the Pimsner-Voiculescu exact sequence in \cite{PV}.

We begin the paper with a preliminary section containing background material about partial-isometric and isometric crossed products, and then identify the spanning elements of the kernel of the natural homomorphism from partial isometric crossed product onto the isometric crossed product of a system $(A,\Gamma^{+},\alpha)$.
In Section 3, we construct a covariant partial-isometric representation of $(A,\Gamma^{+},\alpha)$ in $\L(\ell^{2}(\Gamma^{+},A))$ for which it gives an isomorphism of $A\times_{\alpha}^{\piso}\Gamma^{+}$ onto a full corner of the subalgebra of $\L(\ell^{2}(\Gamma^{+},A))$.
In Section 4, we show that when the semigroup $\Gamma^{+}$ is $\N$ then the kernel of that natural homomorphism is a full corner in the algebra of compact operators on $\ell^{2}(\N,A)$.
We discuss in Section 5, the theory of partial-isometric crossed products for systems by automorphic actions of the semigroups $\Gamma^{+}$.
We show that $A\times_{\alpha}^{\piso}\Gamma^{+}$ is a full corner in the classical crossed product $(B_{\Gamma}\otimes A)\times \Gamma$ of a dynamical system by a group of automorphisms.

\section{Preliminaries}
A \emph{partial isometry} $V$ on a Hilbert space $H$ is an operator which satisfies $\|Vh\|=\|h\|$ for all $h\in (\ker V)^{\perp}$.
A bounded operator $V$ is a partial isometry if and only if $VV^{*}V=V$, and then the adjoint $V^{*}$ is a partial isometry too.
Furthermore the two operators $V^{*}V$ and $VV^{*}$ are the orthogonal projections
on the initial space $(\ker V)^{\perp}$ and the range $VH$ respectively.
So for an element $v$ of a $C^*$-algebra $A$ is called a partial isometry if $vv^{*}v=v$.

A \emph{partial-isometric representation} of $\Gamma^{+}$ on a Hilbert space $H$ is a map $V:\Gamma^{+}\rightarrow B(H)$ such that $V_{s}:=V(s)$
is a partial isometry and $V_{s}V_{t}=V_{s+t}$ for every $s,t\in\Gamma^{+}$.
The product $ST$ of two partial isometries $S$ and $T$ is not always a partial isometry, unless $S^{*}S$ commutes with
$TT^{*}$  (Proposition 2.1 of \cite{LR}).
A partial isometry $S$ is called a power partial isometry if $S^{n}$ is a partial isometry for every $n\in\N$.
So a partial isometric representation of $\N$ is determined by a single power partial isometry $V_{1}$ because $V_{n}=V_{1}^{n}$.
Proposition 3.2 of \cite{LR} says that if $V$ is a partial-isometric representation of $\Gamma^{+}$, then every $V_{s}$ is a power partial isometry, and
$V_{s}V_{s}^{*}$ commutes with $V_{t}V_{t}^{*}$,  $V_{s}^{*}V_{s}$ commutes with $V_{t}^{*}V_{t}$.

A \emph{covariant partial-isometric representation} of $(A,\Gamma^{+},\alpha)$ on a Hilbert space $H$
is a pair $(\pi,V)$ consisting of a non-degenerate representation $\pi:A\rightarrow B(H)$ and a partial-isometric representation $V:\Gamma^{+}\rightarrow B(H)$
which satisfies
\begin{equation}\label{cov}
\pi(\alpha_{s}(a))=V_{s}\pi(a)V_{s}^{*} \quad \mbox{and} \quad V_{s}^{*}V_{s}\pi(a)=\pi(a) V_{s}^{*}V_{s} \quad \mbox{ for } s\in\Gamma^{+}, a\in A.
\end{equation}
Every covariant representation $(\pi,V)$ of $(A,\Gamma^{+},\alpha)$ extends to a covariant representation
$(\overline{\pi},V)$ of $(M(A),\Gamma^{+},\overline{\alpha})$.
Lemma 4.3 of \cite{LR} shows that $(\pi,V)$ is a covariant representation of $(A,\Gamma^{+},\alpha)$ if and only if
\[ \pi(\alpha_{s}(a))V_{s}=V_{s}\pi(a) \quad \mbox{and} \quad V_{s}V_{s}^{*} =\overline{\pi}(\overline{\alpha}_{s}(1)) \quad \mbox{ for } s\in\Gamma^{+}, a\in A.\]

Every system $(A,\Gamma^{+},\alpha)$ admits a nontrivial covariant partial-isometric representation \cite[Example 4.6]{LR}.

\begin{definition}\label{def}
A partial-isometric crossed product of $(A,\Gamma^{+},\alpha)$ is a triple $(B,i_{A},i_{\Gamma^{+}})$ consisting of a $C^{*}$-algebra $B$,
a non-degenerate homomorphism $i_{A}:A\rightarrow B$, and a partial-isometric representation $i_{\Gamma^{+}}:\Gamma^{+}\rightarrow M(B)$  such that
\begin{itemize}
\item[(i)] the pair $(i_{A},i_{\Gamma^{+}})$ is a covariant representation of $(A,\Gamma^{+},\alpha)$ in $B$;
\item[(ii)] for every covariant partial-isometric representation $(\pi,V)$ of $(A,\Gamma^{+},\alpha)$ on a Hilbert space $H$ there is a
non-degenerate representation $\pi\times V$ of $B$ on $H$ which satisfies $(\pi\times V)\circ i_{A}=\pi$ and $\overline{(\pi\times V)}\circ i_{\Gamma^{+}}=V$;
and
\item[(iii)] the $C^*$-algebra $B$ is spanned by $\{i_{\Gamma^{+}}(s)^{*}i_{A}(a)i_{\Gamma^{+}}(t): a\in A, s,t \in\Gamma^{+}\}$.
\end{itemize}
\end{definition}

\begin{remark}
Proposition 4.7 of \cite{LR} shows that such $(B,i_{A},i_{\Gamma^{+}})$ always exists, and it is unique up to isomorphism:
if $(C,j_{A},j_{\Gamma^{+}})$ is a triple that satisfies all properties (i),(ii) and (iii) then there is an isomorphism of $B$ onto $C$ which carries
$(i_{A},i_{\Gamma^{+}})$ into $(j_{A},j_{\Gamma^{+}})$.

We use the standard notation $A\times_{\alpha}\Gamma^{+}$ for the crossed product of $(A,\Gamma^{+},\alpha)$,
and we write $A\times_{\alpha}^{\piso}\Gamma^{+}$ if we want to distinguish it with the other kind of crossed product.

Theorem 4.8 of \cite{LR} asserts that
a covariant representation $(\pi,V)$ of $(A,\Gamma^{+},\alpha)$ on $H$ induces a faithful representation $\pi\times V$ of $A\times_{\alpha}\Gamma^{+}$ if and only if $\pi$ is faithful on $(V_{s}^{*}H)^{\perp}$ for all $s>0$, and this condition is equivalent to say that $\pi$ is faithful on the range of $(1-V_{s}^{*}V_{s})$ for all $s>0$.
\end{remark}

\subsubsection{Isometric crossed products}
The above definition of partial-isometric crossed product is analogous to the one for isometric crossed product: the endomorphisms $\alpha_{s}$ are implemented by partial isometries instead of isometries.

We recall that an \emph{isometric representation} $V$ of $\Gamma^{+}$ on a Hilbert space $H$ is a homomorphism $V:\Gamma^{+}\rightarrow B(H)$ such that each $V_{s}$ is an isometry and $V_{s+t}=V_{s}V_{t}$ for all $s,t\in\Gamma^{+}$.
A pair $(\pi,V)$,  of non degenerate representation $\pi$ of $A$ and an isometric representation $V$ of $\Gamma^{+}$ on $H$, is a
\emph{covariant isometric representations} of $(A,\Gamma^{+},\alpha)$ if $\pi(\alpha_{s}(a))=V_{s}\pi(a)V_{s}^{*}$ for all $a\in A$ and $s\in\Gamma^{+}$.
The \emph{isometric crossed product} $A\times_{\alpha}^{\iso}\Gamma^{+}$ is generated by a universal isometric covariant representation $(i_{A},i_{\Gamma^{+}})$, such that there is a bijection $(\pi,V)\mapsto \pi\times V$ between covariant isometric representations of $(A,\Gamma^{+},\alpha)$ and non degenerate representations of $A\times_{\alpha}^{\iso}\Gamma^{+}$. We make a note that some systems $(A,\Gamma^{+},\alpha)$ may not have a non trivial covariant isometric representations, in which case their isometric crossed products give no information about the systems.

When $\alpha:\Gamma^{+}\rightarrow \End(A)$ is an action of $\Gamma^{+}$ such that every $\alpha_{x}$ is an automorphism of $A$, then every isometry $V_{s}$ in a covariant isometric representation $(\pi,V)$ is a unitary. Thus $A\times_{\alpha}^{\iso}\Gamma^{+}$ is isomorphic to the classical group crossed product $A\times_{\alpha}\Gamma$.
For more general situation, \cite{Adji1,Laca} show that we get, by dilating the system $(A,\Gamma^{+},\alpha)$, a $C^{*}$-algebra $B$ and an action $\beta$ of the group $\Gamma$ by automorphisms of $B$ such that $A\times_{\alpha}^{\iso}\Gamma^{+}$ is isomorphic to the full corner $p(B\times_{\alpha}\Gamma)p$ where $p$ is the unit $1_{M(A)}$ in $B$.

If $(A,\Gamma^{+},\alpha)$ is the distinguished system $(B_{\Gamma^{+}},\Gamma^{+},\tau)$ of the unital $C^*$-algebra
$B_{\Gamma^{+}}:=\overline{\newspan}\{1_{s}\in\ell^{\infty}(\Gamma^{+}): s\in\Gamma^{+}\}$ spanned by the characteristics functions
$1_{s}(x)=\left\{ \begin{array}{ll} 1 & \mbox { if } x\ge s \\ 0 &  \mbox{ if } x< s, \end{array} \right.$
and the action $\tau:\Gamma^{+}\rightarrow \End(B_{\Gamma^{+}})$ is given by the translation on $\ell^{\infty}(\Gamma^{+})$ which satisfies $\tau_{t}(1_{s})=1_{s+t}$.
Then \cite{ALNR} shows that any isometric representation $V$ of $\Gamma^{+}$ induces a unital representation $\pi_{V}:1_{s}\mapsto V_{s}V_{s}^{*}$ of $B_{\Gamma^{+}}$ such that
$(\pi_{V},V)$ is a covariant isometric representation of $(B_{\Gamma^{+}},\Gamma^{+},\tau)$,
and the representation $\pi_{V}\times V$ of $B_{\Gamma^{+}}\times_{\tau}^{\iso}\Gamma^{+}$ is faithful provided all $V_{s}$ are non unitary.
Since the isometric representation given by the Toeplitz representation $T:s\mapsto T_{s}$ of $\Gamma^{+}$ on $\ell^{2}(\Gamma^{+})$ is non unitary,
then $\pi_{T}\times T$ is an isomorphism of $B_{\Gamma^{+}}\times_{\tau}^{\iso}\Gamma^{+}$ onto the Toeplitz algebra $\T(\Gamma)$ .

\bigskip

We consider the two crossed product $(A\times_{\alpha}^{\iso}\Gamma^{+},i_{A},i_{\Gamma^{+}})$ and
$(A\times_{\alpha}^{\piso}\Gamma^{+},j_{A},j_{\Gamma^{+}})$ of a dynamical system $(A,\Gamma^{+},\alpha)$.
The equation $i_{\Gamma^{+}}(s)^{*}i_{\Gamma^{+}}(s)i_{A}(a)=i_{A}(a)i_{\Gamma^{+}}(s)^{*}i_{\Gamma^{+}}(s)$ is automatic because $i_{\Gamma^{+}}$ is an isometric representation of $\Gamma^{+}$.
Therefore we have a covariant partial-isometric representation $(i_{A},i_{\Gamma^{+}})$ of $(A,\Gamma^{+},\alpha)$ in the $C^{*}$-algebra $A\times_{\alpha}^{\iso}\Gamma^{+}$, and the universal property of $A\times_{\alpha}^{\piso}\Gamma^{+}$ gives a non degenerate homomorphism
\[ \phi:=i_{A}\times i_{\Gamma^{+}}:(A\times_{\alpha}^{\piso}\Gamma^{+},j_{A},j_{\Gamma^{+}}) \longrightarrow (A\times_{\alpha}^{\iso}\Gamma^{+},i_{A},i_{\Gamma^{+}}), \]
which satisfies
$\phi(j_{\Gamma^{+}}(x)^{*}j_{A}(a)j_{\Gamma^{+}}(y))=i_{\Gamma^{+}}(x)^{*}i_{A}(a)i_{\Gamma^{+}}(y)$ for all $a\in A$ and $x,y\in\Gamma^{+}$.
Consequently $\phi$ is surjective, and then we have a short exact sequence
\begin{equation}\label{xact}
0\longrightarrow\ker\phi\longrightarrow A\times_{\alpha}^{\piso}\Gamma^{+} \stackrel{\phi}{\longrightarrow} A\times_{\alpha}^{\iso}\Gamma^{+}
\longrightarrow 0.
\end{equation}
In the next proposition, we identify spanning elements for the ideal $\ker\phi$.

\begin{prop}\label{surj}
Suppose $(A,\Gamma^{+},\alpha)$ is a dynamical system.
Then
\begin{equation}\label{kernel-piso-iso}
\ker\phi=\overline{\newspan}\{j_{\Gamma^{+}}(x)^{*}j_{A}(a)(1-j_{\Gamma^{+}}(t)^{*}j_{\Gamma^{+}}(t))j_{\Gamma^{+}}(y) : a\in A,\mbox{ and } x,y,t \in\Gamma^{+} \}.
\end{equation}
\end{prop}

Before we prove this proposition, we first want to show the following lemma.

\begin{lemma}\label{p-jstar} For $t\in\Gamma^{+}$, let $P_{t}$ be the projection $1-j_{\Gamma^{+}}(t)^{*}j_{\Gamma^{+}}(t)$.
Then the set $\{P_{t} : t \in\Gamma^{+}\}$ is a family of increasing projections in the multiplier algebra $M(A\times_{\alpha}^{\piso}\Gamma^{+})$,
which satisfy the following equations: $j_{A}(a)P_{t}=P_{t}j_{A}(a)$ for $a\in A$ and $t\in\Gamma^{+}$,
\[ P_{x} j_{\Gamma^{+}}(y)^{*}=\left\{\begin{array}{ll} 0 & \quad \mbox{ if } x\le y \\
j_{\Gamma^{+}}(y)^{*} P_{x-y} & \quad \mbox{  if } x>y,  \end{array} \right. \quad \text{ and } \quad
P_{x} P_{y} =\left\{\begin{array}{ll} P_{x} \quad \mbox{ if } x\le y \\ P_{y} \quad \mbox{  if  } x> y. \end{array} \right. \]
\end{lemma}
\begin{proof}
For $s\ge t$ in $\Gamma^{+}$,
\begin{align*}
P_{s}-P_{t} & = (1-j_{\Gamma^{+}}(s)^{*}j_{\Gamma^{+}}(s))-(1-j_{\Gamma^{+}}(t)^{*}j_{\Gamma^{+}}(t)) \\
& = j_{\Gamma^{+}}(t)^{*}j_{\Gamma^{+}}(t) - j_{\Gamma^{+}}(s)^{*}j_{\Gamma^{+}}(s) \\
& = j_{\Gamma^{+}}(t)^{*}j_{\Gamma^{+}}(t) - j_{\Gamma^{+}}(t)^{*}j_{\Gamma^{+}}(s-t)^{*}j_{\Gamma^{+}}(s-t)j_{\Gamma^{+}}(t) \\
& = j_{\Gamma^{+}}(t)^{*} P_{s-t} j_{\Gamma^{+}}(t)= j_{\Gamma^{+}}(t)^{*} P_{s-t} P_{s-t} j_{\Gamma^{+}}(t)\\
& = [P_{s-t}j_{\Gamma^{+}}(t)]^{*} [P_{s-t}j_{\Gamma^{+}}(t)].
\end{align*}
So $P_{s}-P_{t}\ge 0$, and hence $P_{s}\ge P_{t}$.

If $x\le y$, then
\begin{align*}
P_{x} j_{\Gamma^{+}}(y)^{*} & = (1-j_{\Gamma^{+}}(x)^{*}j_{\Gamma^{+}}(x))j_{\Gamma^{+}}(x)^{*}j_{\Gamma^{+}}(y-x)^{*}\\
& = [j_{\Gamma^{+}}(x)^{*}-j_{\Gamma^{+}}(x)^{*}j_{\Gamma^{+}}(x)j_{\Gamma^{+}}(x)^{*}]j_{\Gamma^{+}}(y-x)^{*}=0,
\end{align*}
and if $x>y$, we have
\begin{align*}
P_{x} j_{\Gamma^{+}}(y)^{*} & = j_{\Gamma^{+}}(y)^{*}-j_{\Gamma^{+}}(x)^{*}j_{\Gamma^{+}}(x)j_{\Gamma^{+}}(y)^{*} \\
& = j_{\Gamma^{+}}(y)^{*}-j_{\Gamma^{+}}(y)^{*}j_{\Gamma^{+}}(x-y)^{*}j_{\Gamma^{+}}(x-y)j_{\Gamma^{+}}(y)j_{\Gamma^{+}}(y)^{*}  \\
& = j_{\Gamma^{+}}(y)^{*}-j_{\Gamma^{+}}(y)^{*} j_{\Gamma^{+}}(x-y)^{*}j_{\Gamma^{+}}(x-y)\overline{j}_{A}(\overline{\alpha}_{y}(1)) \\
& = j_{\Gamma^{+}}(y)^{*}-j_{\Gamma^{+}}(y)^{*} \overline{j}_{A}(\overline{\alpha}_{y}(1)) j_{\Gamma^{+}}(x-y)^{*}j_{\Gamma^{+}}(x-y)\\
& = j_{\Gamma^{+}}(y)^{*}-[j_{\Gamma^{+}}(y)^{*} j_{\Gamma^{+}}(y)j_{\Gamma^{+}}(y)^{*}] j_{\Gamma^{+}}(x-y)^{*}j_{\Gamma^{+}}(x-y) \\
& = j_{\Gamma^{+}}(y)^{*}P_{x-y}.
\end{align*}

Next we use the equation
\[ j_{\Gamma^{+}}(x)^{*}j_{\Gamma^{+}}(x)j_{\Gamma^{+}}(y)^{*}j_{\Gamma^{+}}(y)=j_{\Gamma^{+}}(\max\{x,y\})^{*}j_{\Gamma^{+}}(\max\{x,y\})
\text{ for any } x,y \in\Gamma^{+}, \]
to see that
\begin{align*}
P_{x}P_{y} & = (1-j_{\Gamma^{+}}(x)^{*}j_{\Gamma^{+}}(x))(1-j_{\Gamma^{+}}(y)^{*}j_{\Gamma^{+}}(y))\\
& = 1-j_{\Gamma^{+}}(x)^{*}j_{\Gamma^{+}}(x)-j_{\Gamma^{+}}(y)^{*}j_{\Gamma^{+}}(y)+j_{\Gamma^{+}}(x)^{*}j_{\Gamma^{+}}(x)j_{\Gamma^{+}}(y)^{*}j_{\Gamma^{+}}(y)\\
& = 1-j_{\Gamma^{+}}(x)^{*}j_{\Gamma^{+}}(x)-j_{\Gamma^{+}}(y)^{*}j_{\Gamma^{+}}(y)+j_{\Gamma^{+}}(\max\{x,y\})^{*}j_{\Gamma^{+}}(\max\{x,y\})\\
& = \left\{\begin{array}{ll} P_{x} \quad \mbox{ if } x\le y \\ P_{y} \quad \mbox{  if  } x> y. \end{array} \right.
\end{align*}
\end{proof}

\begin{proof}[Proof of Proposition \ref{surj}]
We clarify that the right hand side of (\ref{kernel-piso-iso})
\[ \mathcal{I}:=\overline{\newspan}\{j_{\Gamma^{+}}(x)^{*}j_{A}(a)(1-j_{\Gamma^{+}}(t)^{*}j_{\Gamma^{+}}(t))j_{\Gamma^{+}}(y) : a\in A, \mbox{ and } x,y,t \in\Gamma^{+} \}\]
is an ideal of $(A\times_{\alpha}^{\piso}\Gamma^{+},j_{A},j_{\Gamma^{+}})$, by showing that $j_{A}(b)\I$ and $j_{\Gamma^{+}}(s)\I$,
$j_{\Gamma^{+}}(s)^{*}\I$ are contained in $\I$ for all $b\in A$ and $s\in\Gamma^{+}$. The last containment is trivial.
For the first two, we compute using the partial isometric covariance of $(j_{A},j_{\Gamma^{+}})$ to get the following equations for $b\in A$, $s, x\in\Gamma^{+}$:
\[ j_{A}(b)j_{\Gamma^{+}}(x)^{*}=[j_{\Gamma^{+}}(x) j_{A}(b^{*})]^{*}=[j_{A}(\alpha_{x}(b^{*}))j_{\Gamma^{+}}(x)]^{*}=
j_{\Gamma^{+}}(x)^{*}j_{A}(\alpha_{x}(b)), \]
and
\[
j_{\Gamma^{+}}(s)j_{\Gamma^{+}}(x)^{*}=\left\{\begin{array}{ll} j_{\Gamma^{+}}(x-s)^{*}j_{\Gamma^{+}}(x)j_{\Gamma^{+}}(x)^{*}=
j_{\Gamma^{+}}(x-s)^{*}\overline{j}_{A}(\overline{\alpha}_{x}(1))  & \mbox{ if } s<x \\
j_{\Gamma^{+}}(x)j_{\Gamma^{+}}(x)^{*}=\overline{j}_{A}(\overline{\alpha}_{x}(1)) & \mbox{ if } s=x \\
j_{\Gamma^{+}}(s-x)j_{\Gamma^{+}}(x)j_{\Gamma^{+}}(x)^{*}=\overline{j}_{A}(\overline{\alpha}_{s}(1)) j_{\Gamma^{+}}(s-x) & \mbox{ if } s>x.
\end{array}
\right.
\]
Consequently we have
\[
j_{A}(b)j_{\Gamma^{+}}(x)^{*}j_{A}(a)P_{t} j_{\Gamma^{+}}(y)
= j_{\Gamma^{+}}(x)^{*}j_{A}(\alpha_{x}(b)a)P_{t}j_{\Gamma^{+}}(y) \in \I,
\]
and
\[
j_{\Gamma^{+}}(s)j_{\Gamma^{+}}(x)^{*}j_{A}(a)P_{t} j_{\Gamma^{+}}(y)
 = j_{\Gamma^{+}}(x-s)^{*}j_{A}(\overline{\alpha}_{x}(1)a)P_{t} j_{\Gamma^{+}}(y) \in \I
\]
whenever $b\in A$ and $t, s\le x$ in $\Gamma^{+}$.
If $s>x$, then
\[ P_{t}j_{\Gamma^{+}}(s-x)^{*}=\left\{ \begin{array}{ll} 0 & \mbox{ for } t\le s-x \\ j_{\Gamma^{+}}(s-x)^{*}P_{t-(s-x)} & \mbox { for } t>s-x. \end{array}\right. \]
Therefore
\begin{align*} j_{\Gamma^{+}}(s)j_{\Gamma^{+}}(x)^{*}j_{A}(a)P_{t} j_{\Gamma^{+}}(y) & =
\overline{j}_{A}(\overline{\alpha}_{s}(1))j_{\Gamma^{+}}(s-x)j_{A}(a)P_{t} j_{\Gamma^{+}}(y)\\
& = \overline{j}_{A}(\overline{\alpha}_{s}(1))j_{A}(\alpha_{s-x}(a))j_{\Gamma^{+}}(s-x)P_{t}j_{\Gamma^{+}}(y)\\
& = j_{A}(\overline{\alpha}_{s}(1)\alpha_{s-x}(a))~[P_{t}j_{\Gamma^{+}}(s-x)^{*}]^{*}~j_{\Gamma^{+}}(y),
\end{align*}
which is the zero element of $\I$ for $t\le s-x$, and is the element
\[ j_{A}(\overline{\alpha}_{s}(1)\alpha_{s-x}(a))P_{t-(s-x)}j_{\Gamma^{+}}(s-x+y) \text{ of } \I \text{ for } t>s-x. \]
So $j_{\Gamma^{+}}(s)j_{\Gamma^{+}}(x)^{*}j_{A}(a)P_{t} j_{\Gamma^{+}}(y)$ belongs to $\I$, and
$\I$ is an ideal of $A\times_{\alpha}^{\piso}\Gamma^{+}$.

We are now showing the equation $\ker \phi=\I$. The first inclusion $\I\subset \ker\phi$ follows from the fact that $\I$ is an ideal of $A\times_{\alpha}^{\piso}\Gamma^{+}$, and that $\overline{\phi}(P_{t})=1-i_{\Gamma^{+}}(t)^{*}i_{\Gamma^{+}}(t)=0$ for all $t\in\Gamma^{+}$.
For the other inclusion, suppose $\rho$ is a non degenerate representation of
$A\times_{\alpha}^{\piso}\Gamma^{+}$ on a Hilbert space $H$ with $\ker\rho=\I$.
Then the pair $(\pi:=\rho\circ j_{A}, V:=\overline{\rho}\circ j_{\Gamma^{+}})$ is a covariant partial-isometric representation of $(A,\Gamma^{+},\alpha)$ on $H$. We claim that every $V_{t}$ is an isometry. To see this,
let $(a_{\lambda})$ be an approximate identity for $A$. Then we have
\[
 0 =\rho(j_{A}(a_{\lambda})(1-j_{\Gamma^{+}}(t)^{*}j_{\Gamma^{+}}(t)))
 =\pi(a_{\lambda})(1-V_{t}^{*}V_{t}) \text{ for all } \lambda,
\]
and $\pi(a_{\lambda})(1-V_{t}^{*}V_{t})$ converges strongly to $1-V_{t}^{*}V_{t}$ in $B(H)$.
Therefore $1-V_{t}^{*}V_{t}=0$.
Consequently the pair $(\pi,V)$ is a covariant isometric representation
of $(A,\Gamma^{+},\alpha)$ on $H$, and
hence there exists  a non degenerate representation $\psi$ of $(A\times_{\alpha}^{\iso}\Gamma^{+},i_{A},i_{\Gamma^{+}})$ on $H$ which satisfies
$\psi(i_{A}(a))=\rho(j_{A}(a))$ and $\overline{\psi}(i_{\Gamma^{+}}(x))=\overline{\rho}(j_{\Gamma^{+}}(x))$ for all $a\in A$ and $x\in\Gamma^{+}$.
So $\psi\circ\phi=\rho$ on the spanning elements of $A\times_{\alpha}^{\piso}\Gamma^{+}$, thus $\ker\phi\subset\ker\rho$.
\end{proof}

\begin{prop}
If $\Gamma$ is a subgroup of $\R$, then $\ker \phi$ is an essential ideal of the crossed product $A\times_{\alpha}^{\piso}\Gamma^{+}$.
\end{prop}

\begin{proof}
Let $J$ be a non zero ideal of $A\times_{\alpha}^{\piso}\Gamma^{+}$, we want to show that $J\cap \ker\phi\neq \{0\}$.
Assume that $\ker\phi\neq \{0\}$.
Take a non degenerate representation $\pi\times V$ of $A\times_{\alpha}^{\piso}\Gamma^{+}$ on $H$ such that  $\ker\pi\times V= J$.
Since $J\neq \{0\}$, $\pi\times V$ is not a faithful representation.
Consequently, by \cite[Theorem 4.8]{LR}, $\pi$ does not act faithfully on $(V_{s}^{*}H)^{\perp}$ for some $s\in \Gamma^{+}\backslash\{0\}$ .
So there is $a\neq 0$ in $A$ such that $\pi(a)(1-V_{s}^{*}V_{s})=0$.
It follows from
\[ 0=\pi(a)(1-V_{s}^{*}V_{s})=\pi\times V (j_{A}(a)(1-j_{\Gamma^{+}}(s)^{*}j_{\Gamma^{+}}(s))), \]
that $j_{A}(a)(1-j_{\Gamma^{+}}(s)^{*}j_{\Gamma^{+}}(s))$ belongs to $\ker\pi\times V=J$.
Moreover $j_{A}(a)(1-j_{\Gamma^{+}}(s)^{*}j_{\Gamma^{+}}(s))$ is also contained in $\ker\phi$ because $\overline{\phi}(P_{s})=0$,
hence it is contained in $\ker\phi \cap J$.

Next we have to clarify that $j_{A}(a)(1-j_{\Gamma^{+}}(s)^{*}j_{\Gamma^{+}}(s))$ is nonzero.
If it is zero, then $1-j_{\Gamma^{+}}(s)^{*}j_{\Gamma^{+}}(s)=0$ because $j_{A}(a)\neq 0$ by injectivity of $j_{A}$.
Thus $j_{\Gamma^{+}}(s)$ is an isometry, and so is $j_{\Gamma^{+}}(ns)$ for every $n\in\N$.
We claim that every $j_{\Gamma^{+}}(x)$ is an isometry, and consequently
$A\times_{\alpha}^{\piso}\Gamma^{+}$ is isomorphic to  $A\times_{\alpha}^{\iso}\Gamma^{+}$.
Therefore $\ker\phi =0$, and $j_{A}(a)(1-j_{\Gamma^{+}}(s)^{*}j_{\Gamma^{+}}(s))$ can not be zero.

To justify the claim, note that if $x<s$ then $s-x<s$, and we have
\begin{align*}
j_{\Gamma^{+}}(s-x)^{*}  j_{\Gamma^{+}}(s) & = j_{\Gamma^{+}}(s-x)^{*}j_{\Gamma^{+}}(s-x)j_{\Gamma^{+}}(s-(s-x)) \\
& = [j_{\Gamma^{+}}(s-x)^{*}j_{\Gamma^{+}}(s-x)][j_{\Gamma^{+}}(x)j_{\Gamma^{+}}(x)^{*}]j_{\Gamma^{+}}(x) \\
& = [j_{\Gamma^{+}}(x)j_{\Gamma^{+}}(x)^{*}][j_{\Gamma^{+}}(s-x)^{*}j_{\Gamma^{+}}(s-x)]j_{\Gamma^{+}}(x) \\
& =j_{\Gamma^{+}}(x) j_{\Gamma^{+}}(s)^{*}j_{\Gamma^{+}}(s)=j_{\Gamma^{+}}(x).
\end{align*}
So the equation $j_{\Gamma^{+}}(s)^{*}=j_{\Gamma^{+}}(x)^{*}j_{\Gamma^{+}}(s-x)^{*}$ implies
\[ 1=j_{\Gamma^{+}}(s)^{*}j_{\Gamma^{+}}(s)=j_{\Gamma^{+}}(x)^{*}j_{\Gamma^{+}}(s-x)^{*} j_{\Gamma^{+}}(s)=
j_{\Gamma^{+}}(x)^{*}j_{\Gamma^{+}}(x). \]
Thus $j_{\Gamma^{+}}(x)$ is an isometry for every $x< s$.
For $x>s$, by the Archimedean property of $\Gamma$, there exists $n_{x}\in\N$ such that $x< n_{x} s$, and since $j_{\Gamma^{+}}(n_{x}s)$ is an isometry, applying the previous arguments, we see that $j_{\Gamma^{+}}(x)$ is an isometry.
\end{proof}

\section{The partial-isometric crossed product as a full corner.}
Suppose $(A,\Gamma^{+},\alpha)$  is a dynamical system.
We consider the Hilbert $A$-module
$\ell^{2}(\Gamma^{+},A)=\{f:\Gamma^{+}\rightarrow A:\sum_{x\in\Gamma^{+}} f(x)^{*}f(x) \text{ converges in the norm of }A \}$ with the module structure:  $(f\cdot a)(x) =f(x)a$ and $\langle f,g\rangle=\sum_{x\in\Gamma^{+}} f(x)^{*}g(x)$ for $f,g\in\ell^{2}(\Gamma^{+},A)$ and $a\in A$.
One can also want to consider the Hilbert $A$-module $\ell^{2}(\Gamma^{+})\otimes A$, the completion of
the vector space tensor product $\ell^{2}(\Gamma^{+})\odot A$, that has a right (incomplete) inner product $A$-module structure:
$(x\otimes a)\cdot b=x\otimes ab$ and $\langle x\otimes a,y\otimes b\rangle=(y|x) a^{*}b$ for $x,y\in \ell^{2}(\Gamma^{+})$ and $a, b\in A$.
The two modules are naturally isomorphic via the map defined by $\phi: x\otimes a\mapsto \phi(x\otimes a)(t)=x(t)a$ for $x\in\ell^{2}(\Gamma^{+}), t\in\Gamma^{+}, a\in A$.

Let $\pi_{\alpha}:A\rightarrow \L(\ell^{2}(\Gamma^{+},A))$ be a map of $A$ into the $C^*$-algebra $\L(\ell^{2}(\Gamma^{+},A))$ of adjointable operators on $\ell^{2}(\Gamma^{+},A)$, defined by
\[
 (\pi_{\alpha}(a)f)(t)=\alpha_{t}(a)f(t) \text{ for } a\in A, f\in \ell^{2}(\Gamma^{+},A). \]
It is a well-defined map as we can see that $\pi_{\alpha}(a)f\in \ell^{2}(\Gamma^{+},A)$:
\[ \sum_{t\in\Gamma^{+}}(\alpha_{t}(a)f(t))^{*}(\alpha_{t}(a)f(t))=\sum_{t\in\Gamma^{+}}f(t)^{*}\alpha_{t}(a^{*}a)f(t)\le
\|\alpha_{t}(a^{*}a)\|\sum_{t\in\Gamma^{+}}f(t)^{*}f(t).\]
Moreover $\pi_{\alpha}$ is an injective *-homomorphism, which could be degenerate (for example when each of endomorphism $\alpha_{t}$ acts on a unital algebra $A$ and $\alpha_{t}(1)\neq 1$).

Let $S\in \L(\ell^{2}(\Gamma^{+},A))$ defined by
\[ S_{t}(f)(i)=\left\{ \begin{array}{ll} f(i-t) & \text{ if }   i\ge t  \\ 0  & \text{ if } i<t. \end{array} \right. \]
Then $S_{t}^{*}S_{t}=1$, $S_{t}S_{t}^{*} \neq 1$, and the pair $(\pi_{\alpha},S)$ satisfies the following equations:
\begin{equation}\label{pii-S}
\pi_{\alpha}(a)S_{t}=S_{t}\pi_{\alpha}(\alpha_{t}(a))
\text{ and }
(1-S_{t}S_{t}^{*}) \pi_{\alpha}(a)=\pi_{\alpha}(a)(1-S_{t}S_{t}^{*})  \text{ for all } a\in A, t\in\Gamma^{+}.
\end{equation}

Next we consider the vector subspace of $\L(\ell^{2}(\Gamma^{+},A))$ spanned by $\{S_{x}\pi_{\alpha}(a)S_{y}^{*}: a\in A, x,y\in\Gamma^{+}\}$.
Using the equations in (\ref{pii-S}), one can see this space is closed under the multiplication and adjoint, we therefore have a  $C^*$-subalgebra of  $\L(\ell^{2}(\Gamma^{+},A))$, namely
\begin{equation}\label{t-alpha}
\T_{\alpha}:=\overline{\newspan}\{S_{x}\pi_{\alpha}(a)S_{y}^{*}: a\in A, x,y\in\Gamma^{+}\}.
\end{equation}
One can see that $x\in\Gamma^{+}\mapsto S_{x}\in M(\T_{\alpha})$ is a semigroup of non unitary isometries, and $\pi_{\alpha}(A)\subseteq\T_{\alpha}$.
We show in Lemma \ref{pi-alpha-bar} that $\pi_{\alpha}$ extends to the strictly continuous homomorphism $\overline{\pi}_{\alpha}$ on the multiplier algebra $M(A)$, and the equations in (\ref{pii-S}) remain valid.

The algebra $\T_{\alpha}$ defined in (\ref{t-alpha}) satisfies the following natural properties.
If $(A,\Gamma^{+},\alpha)$ and $(B,\Gamma^{+},\beta)$ are two dynamical systems with extendible endomorphism actions,
let $S_{x}\pi_{\alpha}(a)S_{y}^{*}$ and $T_{x}\pi_{\beta}(b)T_{y}^{*}$ denote spanning elements for $\T_{\alpha}$ and $\T_{\beta}$ respectively.
If $\phi:A\rightarrow B$ is a non degenerate homomorphism such that $\phi\circ \alpha_{t}=\beta_{t}\circ\phi$ for every $t\in\Gamma^{+}$, then
by using the identification $\ell^{2}(\Gamma^{+},A)\otimes_{A} B~\simeq~\ell^{2}(\Gamma^{+},B)$,
we have a homomorphism $\tau_{\phi}:\T_{\alpha}\rightarrow \T_{\beta}$ which satisfies
$\tau_{\phi}(S_{x}\pi_{\alpha}(a)S_{y}^{*})=T_{x}\pi_{\beta}(\phi(a))T_{y}^{*}$ for all $a\in A$ and $x,y\in\Gamma^{+}$.
Note that if $\phi$ is injective then so is $\tau_{\phi}$.
This property is consistent with the extendibility of endomorphism $\alpha_{t}$ and $\beta_{t}$.
Since the canonical map $\iota_{A}:A\rightarrow M(A)$ is injective and non degenerate, it follows that
we have an injective homomorphism $\tau_{\iota_{A}}: \T_{\alpha} \rightarrow \T_{\overline{\alpha}}$ such that $\tau_{\iota_{A}}(\T_{\alpha})$ is an ideal of
$\T_{\overline{\alpha}}$.
Moreover since the non degenerate homomorphism $\phi:A\rightarrow B$ extends to $\overline{\phi}$ on the multiplier algebras which satisfies $\overline{\phi}\circ\overline{\alpha}_{t}=\overline{\beta}_{t}\circ\overline{\phi}~\forall t\in\Gamma^{+}$,
therefore $\overline{\phi}$ induces the homomorphism $\tau_{\overline{\phi}}:\T_{\overline{\alpha}}\rightarrow \T_{\overline{\beta}}$,
and it satisfies $\tau_{\overline{\phi}}\circ \tau_{\iota_{A}}=\tau_{\iota_{B}}\circ \tau_{\phi}$.

\begin{lemma}\label{pi-alpha-bar}
The homomorphism $\pi_{\alpha}:A\rightarrow M(\T_{\alpha})$  extends to the strictly continuous homomorphism $\overline{\pi}_{\alpha}$ on the multiplier algebra $M(A)$, such that the pair $(\overline{\pi}_{\alpha},S)$ satisfies
$\overline{\pi}_{\alpha}(m) S_{t}=S_{t}\overline{\pi}_{\alpha}(\overline{\alpha}_{t}(m))$ and
$(1-S_{t}S_{t}^{*}) \overline{\pi}_{\alpha}(m)=\overline{\pi}_{\alpha}(m)(1-S_{t}S_{t}^{*})$ for all $m\in M(A)$ and $t\in\Gamma^{+}.$
\end{lemma}
\begin{proof}
We want to find a projection $p\in M(\T_{\alpha})$ such that
$\pi_{\alpha}(a_{\lambda})$ converges strictly to $p$ in $M(\T_{\alpha})$ for an approximate identity $(a_{\lambda})$ in $A$.

Consider the map $p$ defined on $\ell^{2}(\Gamma^{+},A)$ by
\[ (p (f))(t)=\overline{\alpha}_{t}(1)f(t). \]
First we clarify that $p(f)$ belongs to $\ell^{2}(\Gamma^{+},A)$ for all $f\in \ell^{2}(\Gamma^{+},A)$.
Let $t\in\Gamma^{+}$, then we have
\begin{align*}
(p(f))(t)^{*}(p(f))(t)=(\overline{\alpha}_{t}(1)f(t))^{*}(\overline{\alpha}_{t}(1)f(t))=f(t)^{*}\overline{\alpha}_{t}(1)f(t).
\end{align*}
Since $\overline{\alpha}_{t}(1)$ is a positive element of $M(A)$, it follows that
\[ f(t)^{*}\overline{\alpha}_{t}(1)f(t)\le \|\overline{\alpha}_{t}(1)\|f(t)^{*} f(t)\le f(t)^{*} f(t). \]
Consequently  $0\le \sum_{t\in F} (p(f))(t)^{*}p(f)(t)\le \sum_{t\in F}f(t)^{*} f(t)$ for every finite set $F\subset \Gamma^{+}$.
Moreover we know that  the sequence of partial sums of $\sum_{t\in\Gamma^{+}} f(t)^{*} f(t)$ is Cauchy  in $A$
because $f\in \ell^{2}(\Gamma^{+},A)$.
Therefore $\sum_{t\in\Gamma^{+}} (p(f))(t)^{*}p(f)(t)$ converges in $A$, and hence $p(f)\in \ell^{2}(\Gamma^{+},A)$.

On can see from the definition of $p$ that it is a linear map, and the computations below show it is adjointable, which particularly it satisfies $p^{*}=p$ and $p^{2}=p$. So $p$ is a projection in $\L(\ell^{2}(\Gamma^{+},A))$:
\begin{align*}
\langle p(f),g\rangle & = \sum_{t\in\Gamma^{+}} (p(f)(t))^{*}g(t) = \sum_{t\in\Gamma^{+}} (\overline{\alpha}_{t}(1)f(t))^{*}g(t)
= \sum_{t\in\Gamma^{+}} f(t)^{*}\overline{\alpha}_{t}(1)g(t) \\
& = \sum_{t\in\Gamma^{+}} f(t)^{*} (p(g)(t)) = \langle f, p(g)\rangle.
\end{align*}
To see that $p$ belongs to $M(\T_{\alpha})$,
a direct computation on every $f\in \ell^{2}(\Gamma^{+},A)$ shows that
$[(p~(S_{x}\pi_{\alpha}(a)S_{y}^{*}))~f](t)=[S_{x}\pi_{\alpha}(\overline{\alpha}_{x}(1)a)S_{y}^{*}~ f](t)$ and
$[((S_{x}\pi_{\alpha}(a)S_{y}^{*})~p)~f](t)=[S_{x}\pi_{\alpha}(a \overline{\alpha}_{y}(1))S_{y}^{*}~ f](t)$.
Thus $p$ multiples every spanning element of $\T_{\alpha}$ into itself, so $p\in M(\T_{\alpha})$.

Now we want to prove that $(\pi_{\alpha}(a_{\lambda}))_{\lambda\in\Lambda}$ converges strictly to $p$ in $M(\T_{\alpha})$.
For this we show that $\pi_{\alpha}(a_{\lambda})S_{x}\pi_{\alpha}(a)S_{y}^{*}$ and $S_{x}\pi_{\alpha}(a)S_{y}^{*}\pi_{\alpha}(a_{\lambda})$ converge in
$\T_{\alpha}$ to $p~S_{x}\pi_{\alpha}(a)S_{y}^{*}$ and $S_{x}\pi_{\alpha}(a)S_{y}^{*}~p$ respectively.
Note that $\pi_{\alpha}(a_{\lambda})S_{x}\pi_{\alpha}(a)S_{y}^{*}=S_{x}\pi_{\alpha}(\alpha_{x}(a_{\lambda})a)S_{y}^{*}\in\T_{\alpha}$ and
$S_{x}\pi_{\alpha}(a)S_{y}^{*}\pi_{\alpha}(a_{\lambda})=S_{x}\pi_{\alpha}(a\alpha_{y}(a_{\lambda}))S_{y}^{*} \in \T_{\alpha}$.
Since $\alpha_{x}(a_{\lambda})a\rightarrow \overline{\alpha}_{x}(1)a$  in $A$ by the extendibility of $\alpha_{x}$, it follows that
$S_{x}\pi_{\alpha}(\alpha_{x}(a_{\lambda})a)S_{y}^{*} \rightarrow S_{x}\pi_{\alpha}(\overline{\alpha}_{x}(1)a) S_{y}^{*}=p~(S_{x}\pi_{\alpha}(a)S_{y}^{*})$ and $S_{x}\pi_{\alpha}(a\alpha_{y}(a_{\lambda}))S_{y}^{*}\rightarrow S_{x}\pi_{\alpha}(a\overline{\alpha}_{y}(1))S_{y}^{*}=
(S_{x}\pi_{\alpha}(a)S_{y}^{*})~p$ in $\T_{\alpha}$.
Thus we have shown that $\pi_{\alpha}$ is extendible, and therefore we have $\overline{\pi}_{\alpha}(1_{M(A)})=p$.

Next we want to clarify the equation $\overline{\pi}_{\alpha}(m)S_{x}=S_{x}\overline{\pi}_{\alpha}(\overline{\alpha}_{x}(m))$ in $M(\T_{\alpha})$.
Let $(a_{\lambda})$ be an approximate identity for $A$. The extendibility of $\pi_{\alpha}$ implies $\pi_{\alpha}(a_{\lambda}m)\rightarrow \overline{\pi}_{\alpha}(m)$ strictly in $M(\T_{\alpha})$, and hence
$\pi_{\alpha}(a_{\lambda}m)S_{x}\rightarrow \overline{\pi}_{\alpha}(m)S_{x}$ strictly in $M(\T_{\alpha})$.
But $\pi_{\alpha}(a_{\lambda}m)S_{x}=S_{x}\pi_{\alpha}(\alpha_{x}(a_{\lambda}m))$ converges strictly to $S_{x}\overline{\pi}_{\alpha}(\overline{\alpha}_{x}(m))$ in $M(\T_{\alpha})$.
Therefore $\overline{\pi}_{\alpha}(m)S_{x}=S_{x}\overline{\pi}_{\alpha}(\overline{\alpha}_{x}(m))$.
Similar arguments show that
$\overline{\pi}_{\alpha}(m)(1-S_{t}S_{t}^{*})=(1-S_{t}S_{t}^{*})\overline{\pi}_{\alpha}(m)$ in $M(\T_{\alpha})$.
\end{proof}

We have already shown that $\pi_{\alpha}:A\rightarrow M(\T_{\alpha})$ is extendible in Lemma \ref{pi-alpha-bar}.
Therefore we have a projection $\overline{\pi}_{\alpha}(1_{M(A)})=p$ in $M(\T_{\alpha})$.
Note that $p$ is the identity of $pM(\T_{\alpha})p$, and
$\pi_{\alpha}(a)=\pi_{\alpha}(1_{M(A)}a1_{M(A)})=p\pi_{\alpha}(a)p\in pM(\T_{\alpha})p$.
We claim that the homomorphism $\pi_{\alpha}:A\rightarrow pM(\T_{\alpha})p$ is non degenerate.
To see this, let $(a_{\lambda})$ be an approximate identity for $A$, and $\xi:=S_{x}\pi_{\alpha}(b)S_{y}^{*}$.
Then $\pi_{\alpha}(a_{\lambda})p\xi p=S_{x}\pi_{\alpha}(\alpha_{x}(a_{\lambda})b)S_{y}^{*}p$ converges to
$S_{x}\pi_{\alpha}(\overline{\alpha}_{x}(1)b)S_{y}p=p\xi p$ in $p\T_{\alpha}p$.
Similar arguments show that $p\xi p \pi_{\alpha}(a_{\lambda})\rightarrow p\xi p$ in $p\T_{\alpha}p$.

In the next proposition we show that the algebra $p\T_{\alpha}p$ is a partial-isometric crossed product of $(A,\Gamma^{+},\alpha)$.

\begin{prop}\label{piso-ptp}
Suppose $(A,\Gamma^{+},\alpha)$ is a system such that every $\alpha_{x}\in \End(A)$ is extendible.
Let $p=\overline{\pi}_{\alpha}(1_{M(A)})$, and let
\[ k_{A}: A \rightarrow  p\T_{\alpha}p \quad \text{and} \quad w:\Gamma^{+} \rightarrow M(p\T_{\alpha}p)\]
be the maps defined by
\[ k_{A}(a)= \pi_{\alpha}(a) \quad \text{and} \quad w_{x}=pS_{x}^{*}p. \]
Then the triple $(p\T_{\alpha}p,k_{A},w)$ is a partial-isometric crossed product of $(A,\Gamma^{+},\alpha)$, and
therefore $\psi:=k_{A}\times w: (A\times_{\alpha}^{\piso}\Gamma^{+},i_{A},v) \rightarrow p\T_{\alpha}p$ is an isomorphism which satisfies
$\psi(i_{A}(a))=k_{A}(a)$ and $\psi(v_{x})=w_{x}$.
Moreover, the crossed product $A\times_{\alpha}^{\piso}\Gamma^{+}$ is Morita equivalent to the algebra $\T_{\alpha}$.
\end{prop}

Before we prove the proposition, we show the following lemma.

\begin{lemma}\label{phi-injective}
The pair $(k_{A},w)$ forms a covariant partial-isometric representation of  $(A,\Gamma^{+},\alpha)$ in $p\T_{\alpha}p$, and that
the homomorphism $\varphi:= k_{A}\times w: A\times_{\alpha}^{\piso}\Gamma^{+} \rightarrow p\T_{\alpha}p$ is injective.
\end{lemma}
\begin{proof}
Each of $w_{x}$ is a partial isometry: $w_{x}=pS_{x}^{*}p=\overline{\pi}_{\alpha}(\overline{\alpha}_{x}(1))S_{x}^{*}\Rightarrow
w_{x} w_{x}^{*} w_{x}=\overline{\pi}_{\alpha}(\overline{\alpha}_{x}(1))S_{x}^{*}=w_{x}$,
and
\[ w_{x}w_{y} = \overline{\pi}_{\alpha}(\overline{\alpha}_{x}(1))S_{x}^{*}\overline{\pi}_{\alpha}(\overline{\alpha}_{y}(1))S_{y}^{*}
 = \overline{\pi}_{\alpha}(\overline{\alpha}_{x}(1))\overline{\pi}_{\alpha}(\overline{\alpha}_{x+y}(1))S_{x+y}^{*}
 = w_{x+y} \quad \text{for } x,y \in\Gamma^{+}.
\]
The computations below show that $(k_{A},w)$ satisfies the partial-isometric covariance relations:
\begin{align*}
w_{x}k_{A}(a)w_{x}^{*} & =\overline{\pi}_{\alpha}(\overline{\alpha}_{x}(1))S_{x}^{*}[\pi_{\alpha}(a)S_{x}]\overline{\pi}_{\alpha}(\overline{\alpha}_{x}(1)) \\
 & = \overline{\pi}_{\alpha}(\overline{\alpha}_{x}(1))\pi_{\alpha}(\alpha_{x}(a))\overline{\pi}_{\alpha}(\overline{\alpha}_{x}(1))
 = \pi_{\alpha}(\alpha_{x}(a)) =k_{A}(\alpha_{x}(a)),
\end{align*}
and
\begin{align*}
w_{x}^{*}w_{x}k_{A}(a) & =S_{x}\overline{\pi}_{\alpha}(\overline{\alpha}_{x}(1))S_{x}^{*}\pi_{\alpha}(a)
 = S_{x} \pi_{\alpha}(\overline{\alpha}_{x}(1)\alpha_{x}(a))S_{x}^{*} \\
& = S_{x} \pi_{\alpha}(\alpha_{x}(a)\overline{\alpha}_{x}(1))S_{x}^{*}
 =S_{x} \pi_{\alpha}(\alpha_{x}(a))\overline{\pi}_{\alpha}(\overline{\alpha}_{x}(1))S_{x}^{*}\\
 & = \pi_{\alpha}(a)S_{x}\overline{\pi}_{\alpha}(\overline{\alpha}_{x}(1))S_{x}^{*}= \pi_{\alpha}(a)w_{x}^{*}w_{x}= k_{A}(a)w_{x}^{*}w_{x}.
\end{align*}

So we get a non degenerate homomorphism $\varphi:=k_{A}\times w: A\times_{\alpha}^{\piso}\Gamma^{+} \rightarrow p\T_{\alpha}p$.
We want to see it is injective.
Put $p\T_{\alpha}p$ by a faithful and non degenerate representation $\gamma$ into a Hilbert space $H$.
Then we want to prove that the representation $\gamma\circ \varphi$ of $(A\times_{\alpha}^{\piso}\Gamma^{+},i_{A},v)$ on $H$ is faithful.
Let $\sigma=\gamma\circ \varphi\circ i_{A}$ and $t=\overline{\gamma\circ \varphi}\circ v$.
By \cite[Theorem 4.8]{LR}, we have to show that $\sigma$ acts faithfully on the range of $(1-t_{x}^{*}t_{x})$ for every $x>0$ in $\Gamma^{+}$.
If $x>0$ in $\Gamma^{+}$, $a\in A$, and $\sigma(a)|_{{\rm range}(1-t_{x}^{*}t_{x})}=0$,
then we want to see that $a=0$.
First note that $\sigma(a)(1-t_{x}^{*}t_{x})=\gamma\circ\varphi(i_{A}(a)(1-v_{x}^{*}v_{x}))$, and
\begin{align*}
\varphi(i_{A}(a)(1-v_{x}^{*}v_{x})) & =\varphi(i_{A}(a))(\overline{\varphi}(1)-\varphi(v_{x}^{*}v_{x}))
=\varphi(i_{A}(a))(p-\overline{\varphi}(v_{x}^{*})\overline{\varphi}(v_{x})) \\
& = k_{A}(a)(p-w_{x}^{*}w_{x}) \\
& = \pi_{\alpha}(a)(\overline{\pi}_{\alpha}(1)-S_{x}\overline{\pi}_{\alpha}(\overline{\alpha}_{x}(1))\overline{\pi}_{\alpha}(\overline{\alpha}_{x}(1))S_{x}^{*})\\
& =\pi_{\alpha}(a)(\overline{\pi}_{\alpha}(1)-\overline{\pi}_{\alpha}(1)S_{x}S_{x}^{*}\overline{\pi}_{\alpha}(1)) \\
& =\pi_{\alpha}(a)(1-S_{x}S_{x}^{*})\overline{\pi}_{\alpha}(1) =\pi_{\alpha}(a)\overline{\pi}_{\alpha}(1)(1-S_{x}S_{x}^{*}) \\
& =\pi_{\alpha}(a)(1-S_{x}S_{x}^{*}).
\end{align*}
So $\sigma(a)(1-t_{x}^{*}t_{x})=0$ implies $\pi_{\alpha}(a)(1-S_{x}S_{x}^{*})=0$ in $\L(\ell^{2}(\Gamma^{+},A))$.
But for $f\in \ell^{2}(\Gamma^{+},A)$, we have
\[ ((1-S_{x}S_{x}^{*})f)(y)=\left\{\begin{array}{ll} 0 & \text{ for } y\ge x >0 \\ f(y) & \text{ for } y< x. \end{array}\right.
\]
Thus evaluating the operator $\pi_{\alpha}(a)(1-S_{x}S_{x}^{*})$ on a chosen element $f\in \ell^{2}(\Gamma^{+},A)$ where
$f(y)=a^{*}$ for $y=0$ and $f(y)=0$ for $y\neq 0$,
we get
\begin{align*}
(\pi_{\alpha}(a)(1-S_{x}S_{x}^{*})(f))(y) & =
\left\{\begin{array}{ll} \alpha_{y}(a)f(y) & \text{ for } y=0 \\ 0 & \text{ for } y\neq 0 \end{array}\right.
 = \left\{\begin{array}{ll} aa^{*} & \text{ for } y=0 \\ 0 & \text{ for } y\neq 0. \end{array}\right.
\end{align*}
Therefore $aa^{*}=0\in A$, and hence $a=0$.
\end{proof}

\begin{proof}[Proof of Proposition \ref{piso-ptp}]
Let $(\rho,W)$ be a covariant partial-isometric representation of $(A,\Gamma^{+},\alpha)$ on a Hilbert space $H$.
We want to construct a non degenerate representation $\Phi$ of $p\T_{\alpha}p$ on $H$ such that
$\Phi(p S_{i}\pi_{\alpha}(a)S_{j}^{*}p)=W_{i}^{*}\rho(a)W_{j}$ for all $a\in A,i,j\in\Gamma^{+}$.
It follows from this equation that $\Phi(k_{A}(a))=\rho(a)$ for all $a\in A$, and
$\overline{\Phi}(w_{i})=W_{i}$ for $i\in\Gamma^{+}$ because
$\Phi(p\pi_{\alpha}(a_{\lambda})S_{i}^{*}p)=\rho(a_{\lambda})W_{i}$ for all $i\in\Gamma^{+}$, $\rho(a_{\lambda})W_{i}$ converges strongly to $W_{i}$ in $B(H)$,
and
\begin{align*}
\Phi(p\pi_{\alpha}(a_{\lambda})S_{i}^{*}p) & =\Phi(\pi_{\alpha}(a_{\lambda}))\overline{\Phi}(pS_{i}^{*}p)  =\rho(a_{\lambda})\overline{\Phi}(pS_{i}^{*}p)
 \rightarrow \overline{\Phi}(pS_{i}^{*}p)  \text{ strongly in } B(H).
\end{align*}

So we want the representation $\Phi$ to satisfy
\[ \Phi\left(\sum \lambda_{i,j} p S_{i}\pi_{\alpha}(a_{i,j})S_{j}^{*}p\right) =
\sum \lambda_{i,j}\Phi(p S_{i}\pi_{\alpha}(a_{i,j})S_{j}^{*}p)=
\sum \lambda_{i,j}W_{i}^{*}\rho(a_{i,j})W_{j}.\]
We prove that this formula gives a well-defined linear map $\Phi$ on
$\newspan\{pS_{i}\pi_{\alpha}(a)S_{j}^{*}p : a\in A, i,j \in\Gamma^{+}\}$, and simultaneously $\Phi$ extends to $p\T_{\alpha}p$ by
showing that
\[ \left\|\sum \lambda_{i,j}W_{i}^{*}\rho(a_{i,j})W_{j}\right\|\le \left\|\sum \lambda_{i,j} p S_{i}\pi_{\alpha}(a_{i,j})S_{j}^{*}p\right\|. \]
Note that the non degenerate representation $\rho\times W$ of $(A\times_{\alpha}^{\piso}\Gamma^{+},i_{A},v)$ on $H$ satisfies
$\rho\times W(v_{i}^{*}i_{A}(a)v_{j})=W_{i}^{*}\rho(a)W_{j}$, and the injective homomorphism $\varphi:(A\times_{\alpha}^{\piso}\Gamma^{+},i_{A},v)\rightarrow p\T_{\alpha}p$ in Lemma \ref{phi-injective}
satisfies $\varphi(v_{i}^{*}i_{A}(a)v_{j})=w_{i}^{*}k_{A}(a)w_{j}=pS_{i}\pi_{\alpha}(a)S_{j}^{*}p$.
Now we compute
\begin{align*}
\left\|\sum_{i,j\in\Gamma^{+}} \lambda_{i,j}W_{i}^{*}\rho(a_{i,j})W_{j}\right\| & = \left\|\rho\times W\left(\sum \lambda_{i,j}v_{i}^{*}i_{A}(a_{i,j})v_{j}\right)\right\| \\
& \le \left\|\sum \lambda_{i,j}v_{i}^{*}i_{A}(a_{i,j})v_{j}\right\| \\
& =   \left\|\varphi\left(\sum\lambda_{i,j}v_{i}^{*}i_{A}(a_{i,j})v_{j}\right)\right\| \quad \text{by injectivity of } \varphi \\
& =  \left\|\sum \lambda_{i,j}pS_{i}\pi_{\alpha}(a_{i,j})S_{j}^{*}p\right\|.
\end{align*}

Next we verify that $\Phi$ is a *-homomorphism.
It certainly preserves the adjoint, and we claim by our arguments below that it also preserves the multiplication.
Note that
\[ \xi:= (pS_{i}\pi_{\alpha}(a)S_{j}^{*}p)~(pS_{n}\pi_{\alpha}(b)S_{m}^{*}p)
=\left\{\begin{array}{ll}
pS_{i}\pi_{\alpha}(a\overline{\alpha}_{j}(1)b)S_{m}^{*}p &  \text{ for } j=n \\
pS_{i}\pi_{\alpha}(a\alpha_{j-n}(\overline{\alpha}_{n}(1)b))S_{j-n+m}^{*}p &  \text{ for } j>n \\
pS_{i+n-j}\pi_{\alpha}(\alpha_{n-j}(a)\overline{\alpha}_{n}(1)b)S_{m}^{*} p & \text{ for } j<n.
\end{array}\right.
\]
Then use the covariance of $(\rho,W)$ to see that $\Phi(\xi)=(W_{i}^{*}\rho(a)W_{j})(W_{n}^{*}\rho(a)W_{m})$ for all cases of $j$ and $n$.
So $\Phi$ preserves the multiplication. Thus $\Phi$ is a representation of $p\T_{\alpha}p$ on $H$.

We want to see that $\Phi$ is non degenerate.
The representation $\rho$ of $A$ is non degenerate and $\rho(a)=\Phi(\pi_{\alpha}(a))$, therefore
\begin{align*}
H & = \overline{\newspan}\{\rho(a)h: a\in A, h\in H\}
& \subset
\overline{\newspan}\{\Phi(pS_{i}\pi_{\alpha}(a)S_{j}^{*}p)h:a\in A,i,j\in\Gamma^{+},h\in H\},
\end{align*}
so $\Phi$ is non-degenerate.
The $C^*$-algebra $p\T_{\alpha}p$ is spanned by $\{w_{i}^{*}i_{A}(a)w_{j}: a\in A,i,j \in\Gamma^{+}\}$ because
$w_{i}^{*}i_{A}(a)w_{j}=pS_{i}p\pi_{\alpha}(a)pS_{j}^{*}p=pS_{i}\pi_{\alpha}(a)S_{j}^{*}p$.
Thus $p\T_{\alpha}p$ and $A\times_{\alpha}^{\piso}\Gamma^{+}$ are isomorphic.

Finally we prove the fullness of $A\times_{\alpha}^{\piso}\Gamma^{+}$ in $\T_{\alpha}$.
It is enough by \cite[Example 3.6]{RW} to show that $\T_{\alpha}p\T_{\alpha}$ is dense in $\T_{\alpha}=\overline{\newspan}\{S_{i}\pi_{\alpha}(a)S_{j}^{*} : i,j\in\Gamma^{+}, a\in A\}$.
Take a spanning element $S_{i}\pi_{\alpha}(a)S_{j}^{*}\in \T_{\alpha}$ and an approximate identity $(a_{\lambda})$ for $A$.
Then $S_{i}\pi_{\alpha}(a)S_{j}^{*}=\lim_{\lambda} S_{i}\pi_{\alpha}(aa_{\lambda})S_{j}^{*}$, and since
$S_{i}\pi_{\alpha}(aa_{\lambda})S_{j}^{*}=S_{i}\pi_{\alpha}(a)S_{0}^{*}pS_{0}\pi_{\alpha}(a_{\lambda})S_{j}^{*}\in \T_{\alpha}p\T_{\alpha}$, therefore
a linear combination of spanning elements in $\T_{\alpha}$ can be approximated by elements of $\T_{\alpha}p\T_{\alpha}$.
Thus $\overline{\T_{\alpha}p\T_{\alpha}}=\T_{\alpha}$.
\end{proof}

\begin{remark}
When dealing with systems $(A,\Gamma^{+},\alpha)$ in which $\overline{\alpha}_{t}(1)=1$, then
$p=\overline{\pi}_{\alpha}(1)$ is the identity of $\L(\ell^{2}(\Gamma^{+},A))$, and the assertion of
Proposition \ref{piso-ptp} says that $A\times_{\alpha}^{\piso}\Gamma^{+}$ is isomorphic to $\T_{\alpha}$.
\end{remark}

\section{The partial-isometric crossed product of a system by a single endomorphism.}
In this section we consider a system $(A,\N,\alpha)$ of a (non unital) $C^*$-algebra $A$ and an action $\alpha$ of $\N$ by extendible endomorphisms of $A$.
The module $\ell^{2}(\N,A)$ is the vector space of sequences $(x_{n})$ such that the series $\sum_{n\in\N} x_{n}^{*}x_{n}$ converges in the norm of $A$, with the module structure $(x_{n})\cdot a =(x_{n}a)$, and the inner product $\langle(x_{n}),(y_{n})\rangle=\sum_{n\in\N} x_{n}^{*}y_{n}$.

The homomorphism $\pi_{\alpha}:A\rightarrow \L(\ell^{2}(\N,A))$ defined by $\pi_{\alpha}(a)(x_{n})=(\alpha_{n}(a)x_{n})$ is injective, and together with
the non unitary isometry $S\in \L(\ell^{2}(\N,A))$
\[ S(x_{0},x_{1},x_{2},\cdots)=(0,x_{0},x_{1},x_{2},\cdots) \]
satisfy the following equation
\begin{equation}\label{pi-S}
\pi_{\alpha}(a) S_{i} = S_{i}\pi_{\alpha}(\alpha_{i}(a))\quad \text{for all } a\in A, i\in\N.
\end{equation}
Note that $S_{n}\pi_{\alpha}(ab^{*})(1-SS^{*})S_{m}^{*}=\theta_{f,g}$ where
$f(n)=a$ and $f(i)=0$ for $i\neq n$, $g(m)=b$ and $g(i)=0$ for $i\neq m$.
So we can identify  the $C^*$-algebra $\K(\ell^{2}(\N,A))$ as
\[  \overline{\newspan}\{S_{n}\pi_{\alpha}(ab^{*})(1-SS^{*})S_{m}^{*}: n,m\in\N, a,b\in A\}. \]

Let $(A\times_{\alpha}^{\iso}\N,j_{A},T)$ be the isometric crossed product of $(A,\N,\alpha)$, and consider
the natural homomorphism $\phi=(i_{A}\times T) :A\times_{\alpha}^{\piso}\N\rightarrow A\times_{\alpha}^{\iso}\N$.
From the Proposition \ref{surj}, we know that
\begin{equation}\label{kernel-span2}
\ker \phi=\overline{\newspan}\{v_{m}^{*}i_{A}(a)(1-v^{*}v)v_{n}: a\in A, m,n\in\N\}.
\end{equation}

We show in the next theorem that the ideal $\ker \phi$ is a corner in $A\otimes K(\ell^{2}(\N))$.

\begin{theorem}\label{ext}
Suppose $(A,\N,\alpha)$ is a dynamical system in which every $\alpha_{n}:=\alpha^{n}$ extends to a strictly continuous endomorphism
on the multiplier algebra $M(A)$ of $A$.
Let $p=\overline{\pi}_{\alpha}(1_{M(A)}) \in \L(\ell^{2}(\N,A))$.
Then the isomorphism $\psi:A\times_{\alpha}^{\piso}\N\rightarrow p\T_{\alpha}p$ in Proposition \ref{piso-ptp} takes the ideal $\ker \phi$ of $A\times_{\alpha}^{\piso}\N$ given by {\rm (\ref{kernel-span2})} isomorphically to the full corner $p~[K(\ell^{2}(\N,A))]~p$.
So there is a short exact sequence of $C^*$-algebras

\begin{equation}\label{diagram1}
\begin{diagram}\dgARROWLENGTH=0.5\dgARROWLENGTH
\node{0} \arrow{e}
\node{p~[K(\ell^{2}(\N,A))]~p} \arrow{e,t}{\Psi}
\node{A\times_{\alpha}^{\piso}\N} \arrow{e,t}{\phi}
\node{A\times_{\alpha}^{\iso}\N} \arrow{e}
\node{0,}
\end{diagram}
\end{equation}
where $\Psi(pS_{m}\pi_{\alpha}(a)(1-SS^{*})S_{n}^{*}p)=v_{m}^{*}i_{A}(a)(1-v^{*}v)v_{n}$.
\end{theorem}

\begin{proof}
We compute the image $\psi(\mu)$ of a spanning element $\mu:=v_{m}^{*}i_{A}(a)(1-v^{*}v)v_{n}$ of $\ker\phi$
\[
\psi(\mu)  = pS_{m}p\pi_{\alpha}(a)\psi(1-v^{*}v)pS_{n}^{*}p = pS_{m}\pi_{\alpha}(a)(p-pSpS^{*}p)pS_{n}^{*}p,
\]
\[ pSpS^{*}=(\overline{\pi}_{\alpha}(1)S)\overline{\pi}_{\alpha}(1)S^{*}=S\overline{\pi}_{\alpha}(\overline{\alpha}(1))S^{*}=
S(S\overline{\pi}_{\alpha}(\overline{\alpha}(1))^{*}=S(\overline{\pi}_{\alpha}(1)S)^{*}=SS^{*}p\]
and
\[ pS_{n}^{*}p=\overline{\pi}_{\alpha}(1)(\overline{\pi}_{\alpha}(1)S_{n})^{*}  =
\overline{\pi}_{\alpha}(1)(S_{n}\overline{\pi}_{\alpha}(\overline{\alpha_{n}}(1)))^{*}
=\overline{\pi}_{\alpha}(\overline{\alpha_{n}}(1))S_{n}^{*}=(\overline{\pi}_{\alpha}(1)S_{n})^{*}
 =S_{n}^{*}p. \]
Therefore we have
\begin{equation}\label{psi-span}
 \psi(v_{m}^{*}i_{A}(a)(1-v^{*}v)v_{n}) = p~(S_{m}\pi_{\alpha}(a)(1-SS^{*})S_{n}^{*})~p.
\end{equation}
Since $S_{m}\pi_{\alpha}(a)(1-SS^{*})S_{n}^{*}=\lim_{\lambda}S_{m}\pi_{\alpha}(aa_{\lambda}^{*})(1-SS^{*})S_{n}^{*}$
where $(a_{\lambda})$ is an approximate identity in $A$, and
$S_{m}\pi_{\alpha}(aa_{\lambda}^{*})(1-SS^{*})S_{n}^{*}=\theta_{\xi,\eta_{\lambda}}$ for which $\xi,\eta_{\lambda} \in \ell^{2}(\N,A)$ are given by
$\xi(m)=a$ and $\xi(i)=0$ for $i\neq m$, $\eta_{\lambda}(n)=a_{\lambda}$ and $\eta_{\lambda}(i)=0$ for $i\neq n$,
it follows that $\psi(\mu)\in p~[\K(\ell^{2}(\N,A))]~p$. 
Thus $\psi(\ker \phi)\subset p~[K(\ell^{2}(\N,A))]~p$.

Conversely by similar computations to the way we get the equation (\ref{psi-span}), we have
$pS_{m}\pi_{\alpha}(ab^{*})(1-SS^{*})S_{n}^{*}p= \psi(v_{m}^{*}i_{A}(ab^{*})(1-v^{*}v)v_{n})$.
Hence $p~[K(\ell^{2}(\N,A))]~p\subset \psi(\ker \phi)$.
This corner is full because the algebra $K(\ell^{2}(\N,A)) p K(\ell^{2}(\N,A))$ is dense in $K(\ell^{2}(\N,A))$:
for an approximate identity $(a_{\lambda})$ in $A$, we have
\[ S_{m}\pi_{\alpha}(a)(1-SS^{*})S_{n}^{*} = \lim_{\lambda}S_{m}\pi_{\alpha}(aa_{\lambda})(1-SS^{*})S_{n}^{*} \]
and
$S_{m}\pi_{\alpha}(aa_{\lambda})(1-SS^{*})S_{n}^{*} =(S_{m}\pi_{\alpha}(a)(1-SS^{*})S_{0}^{*}) p (S_{0}\pi_{\alpha}(a_{\lambda})(1-SS^{*})S_{n}^{*}$
is contained in $K(\ell^{2}(\N,A)) p K(\ell^{2}(\N,A))$.
\end{proof}

\begin{remark}
The external tensor product $\ell^{2}(\N)\otimes A$ and $\ell^{2}(\N,A)$ are isomorphic as Hilbert $A$-modules \cite[Lemma 3.43]{RW}, and
the isomorphism is given by
\[ \varphi(f\otimes a)(n)=(f(0)a,f(1)a,f(2)a,\cdots) \text{ for }f\in\ell^{2}(\N) \text{ and } a\in A. \]
The isomorphism $\psi: T\in\L(\ell^{2}(\N,A))~\mapsto~\varphi^{-1} T\varphi \in \L(\ell^{2}(\N)\otimes A)$ satisfies
$\psi(\theta_{\xi,\eta})=\varphi^{-1} ~ \theta_{\xi,\eta} ~\varphi=\theta_{\varphi^{-1}(\xi),\varphi^{-1}(\eta)}$
for all $\xi,\eta \in\ell^{2}(\N,A)$.
Therefore $\psi(\K(\ell^{2}(\N,A)))=\K(\ell^{2}(\N)\otimes A)$.
So $\psi(p)=\varphi^{-1} p\varphi=:\tilde{p}$ is a projection in $\L(\ell^{2}(\N)\otimes A)$.
To see how $\tilde{p}$ acts on $\ell^{2}(\N)\otimes A$, let $f\in \ell^{2}(\N)$, $a\in A$ and $\{e_{n}\}$  the usual orthonormal basis in $\ell^{2}(\N)$.
Then $\tilde{p}(f\otimes a)=\varphi^{-1}(p\varphi(f\otimes a))$, and
\[ p\varphi(f\otimes a)=(f(i)\overline{\alpha}_{i}(1)a)_{i\in\N}=\lim_{k\rightarrow\infty} \varphi(\sum_{i=0}^{k}f(i)e_{i}\otimes \overline{\alpha}_{i}(1)a). \]
Therefore
$\tilde{p}(f\otimes a)=\varphi^{-1}(p\varphi(f\otimes a))=\lim_{k\rightarrow\infty}\sum_{i=0}^{k}f(i)e_{i}\otimes\overline{\alpha}_{i}(1)a$, and
hence $p[\K(\ell^{2}(\N,A))]p\simeq \tilde{p}[\K(\ell^{2}(\N)\otimes A)]\tilde{p}$.
\end{remark}

\begin{example}\label{example-LR}
We now want to compare our results with \cite[\S 6]{LR}.
Consider a system consisting of the $C^*$-algebra ${\bf c}:=\overline{\newspan}\{1_{n} : n\in\N\}$ of convergent sequences, and the action $\tau$ of $\N$ generated by the usual forward shift (non unital endomorphism) on ${\bf c}$.
The ideal ${\bf c_{0}}:=\overline{\newspan}\{1_{x}-1_{y} : x< y\in\N\}$, of sequences in ${\bf c}$ convergent to $0$,
is an extendible $\tau$-invariant in the sense of \cite{Adji1,AH}.
So we can also consider the systems $({\bf c_{0}},\N,\tau)$ and $({\bf c}/{\bf c_{0}},\N,\tilde{\tau})$, where
the action $\tilde{\tau}_{n}$ of the quotient ${\bf c}/{\bf c_{0}}$ is given by $\tilde{\tau}_{n}(1_{x}+{\bf c_{0}})=\tau_{n}(1_{x})+{\bf c_{0}}$.
We show that the three rows of exact sequences in \cite[Theorem 6.1]{LR}, are given by applying our results to $({\bf c},\N,\tau)$, $({\bf c_{0}},\N,\tau)$ and $({\bf c}/{\bf c_{0}},\N,\tilde{\tau})$.

The crossed product ${\bf c} \times_{\tau}^{\piso}\N$ of $({\bf c},\N,\tau)$ is, by \cite[Proposition 5.1]{LR}, the universal algebra generated by a power partial isometry $v$: a covariant partial-isometric representation $(i_{c},v)$ of $({\bf c},\N,\tau)$ is defined by $i_{c}(1_{n})=v_{n}v_{n}^{*}$.
Let $p=\pi_{\tau}(1)$ be the projection in $\T_{{\bf c},\tau}$,  and the partial-isometric representation $w:n\mapsto w_{n}=pS_{n}^{*}p$ of $\N$ in $p\T_{{\bf c},\tau}p$ gives a representation $\pi_{w}$ of ${\bf c}$ where $\pi_{w}(1_{x})=w_{x}w_{x}^{*}$, such that $(\pi_{w},w)$ is a covariant partial-isometric representation of $({\bf c},\N,\tau)$ in $p\T_{{\bf c},\tau}p$.
This $\pi_{w}$ is the homomorphism $k_{\bf c}:{\bf c}\rightarrow p\T_{{\bf c},\tau}p$ defined by Proposition \ref{piso-ptp}, and
the covariant representation $(\pi_{w},w)$ is $(k_{\bf c},w)$.
So $\pi_{w}\times w = k_{\bf c}\times w$ is an isomorphism of ${\bf c} \times_{\tau}^{\piso}\N$ onto the $C^*$-algebra $p\T_{{\bf c},\tau}p$.

Moreover, the injective homomorphism $\Psi:p~[K(\ell^{2}(\N,{\bf c}))]~p \rightarrow ({\bf c}\times_{\tau}^{\piso}\N,i_{{\bf c}},v)$ in Theorem \ref{ext} satisfies
\[ \Psi(pS_{i}\pi_{\tau}(1_{n})(1-SS^{*})S_{j}^{*}p)=v_{i}^{*}i_{{\bf c}}(1_{n})(1-v^{*}v)v_{j}=v_{i}^{*}v_{n}v_{n}^{*}(1-v^{*}v)v_{j}, \]
and the latter is a spanning element $g_{i,j}^{n}$ of $\ker\varphi_{T}$ by \cite[Lemma 6.2]{LR}.
Consequently the ideal $p~[K(\ell^{2}(\N,{\bf c}))]~p$, in our Theorem \ref{ext}, is the $C^*$ algebra $\A=\pi^{*}(\ker \varphi_{T})$ of \cite[Proposition~6.9]{LR}, where the homomorphism $\varphi_{T}:{\bf c} \times_{\tau}^{\piso}\N\rightarrow \T(\Z)$ is induced by the Toeplitz representation $n \mapsto T_{n}$.
Now the Toeplitz (isometric) representation $T:n \mapsto T_{n}$ on $\ell^{2}(\N)$ gives the isomorphism of ${\bf c} \times_{\tau}^{\iso}\N$ onto the Toeplitz algebra $\T(\Z)$, and ${\bf c_{0}} \times_{\tau}^{\iso}\N$ onto the algebra $K(\ell^{2}(\N))$ of compact operators on $\ell^{2}(\N)$.
Then the second row exact sequence in \cite[Theorem 6.1]{LR} follows from the commutative diagram:

\begin{equation}\label{diagram-c}
\begin{diagram}\dgARROWLENGTH=0.5\dgARROWLENGTH
\node{0} \arrow{e} \node{p~[K(\ell^{2}(\N, {\bf c}))]~p} \arrow{s,r}{\Psi}\arrow{e,t}{\Psi}
\node{{\bf c} \times_{\tau}^{\piso}\N} \arrow{s,r}{\id}\arrow{e,t}{\phi}
\node{{\bf c} \times_{\tau}^{\iso}\N} \arrow{s,r}{T}\arrow{e} \node{0.}
\\
\node{0} \arrow{e} \node{\ker(\varphi_{T})\stackrel{\pi^{*}}{\simeq}\A} \arrow{e,t}{(\pi^{*})^{-1}}
\node{{\bf c} \times_{\tau}^{\piso}\N} \arrow{e,t}{\varphi_{T}}
\node{\T(\Z)} \arrow{e} \node{0.}
\end{diagram}
\end{equation}
\end{example}

Next we do similarly for $({\bf c_{0}},\N,\tau)$ and $({\bf c}/{\bf c_{0}},\N,\tilde{\tau})$ to get the first and third row exact sequences of diagram (6.1) in \cite[Theorem 6.1]{LR}.
We know from \cite[Theorem~2.2]{AH} that ${\bf c_{0}}\times_{\tau}^{\piso}\N$ embeds in $({\bf c} \times_{\tau}^{\piso}\N,i_{c},v)$ as the ideal
$D=\overline{\newspan}\{v_{i}^{*}i_{c}(1_{s}-1_{t})v_{j} : s<t, i,j \in \N\}$,
such that the quotient $({\bf c}\times_{\tau}^{\piso}\N)/({\bf c_{0}}\times_{\tau}^{\piso}\N)\simeq {\bf c}/{\bf c_{0}}\times_{\tilde{\tau}}^{\piso}\N$.
Then the isomorphism $\Phi$ in \cite[Corollary 3.1]{AH} together with the isomorphism $\pi$ in \cite[Proposition 6.9]{LR} give the relations
${\bf c_{0}}\times_{\tau}^{\piso}\N\stackrel{\Phi}{\simeq}\ker(\varphi_{T^{*}})\stackrel{\pi}{\simeq} \A$, where the homomorphism $\varphi_{T^{*}}:{\bf c} \times_{\tau}^{\piso}\N\rightarrow \T(\Z)$ is associated to partial-isometric representation $n \mapsto T_{n}^{*}$.

Let $q=\overline{\pi}_{\tau}(1_{M({\bf c_{0}})})$ be the projection in $M(\T_{{\bf c_{0}},\tau})$.
Then we have
\[ q~[K(\ell^{2}(\N,{\bf c_{0}}))]~q =\overline{\newspan}\{q S_{i}\pi_{\tau}(1_{m}-1_{m+1})(1-SS^{*})S_{j}^{*} q: i,j\le m\}, \]
and
\[ \xi_{ijm}:=\Psi(q S_{i}\pi_{\tau}(1_{m}-1_{m+1})(1-SS^{*})S_{j}^{*} q) = g_{i,j}^{m}-g_{i,j}^{m+1}=f^{m}_{m-i,m-j} - f^{m+1}_{m-i,m-j}\]
where
$g_{i,j}^{m}$ and $f_{i,j}^{m}$ are defined in \cite[Lemma 6.2]{LR}.
So $\xi_{ijm}$ is, by \cite[Lemma 6.4]{LR}, the spanning element of the ideal $\I:=\ker(\varphi_{T^{*}})\cap \ker(\varphi_{T})$.
We use the isomorphism $\pi$ given by \cite[Proposition 6.5]{LR} to identify $\I$ with $\A_{0}$, to have the commutative diagram:
\begin{equation}\label{diagram-c0}
\begin{diagram}\dgARROWLENGTH=0.3\dgARROWLENGTH
\node{0} \arrow{e} \node{q~[K(\ell^{2}(\N, {\bf c_{0}}))]~q} \arrow{s,r}{\Psi} \arrow{e,t}{\Psi}
\node{{\bf c_{0}} \times_{\tau}^{\piso}\N} \arrow{s,r}{\Phi}\arrow{e,t}{\phi}
\node{{\bf c_{0}} \times_{\tau}^{\iso}\N} \arrow{s,r}{T} \arrow{e} \node{0}
\\
\node{0} \arrow{e} \node{\I\stackrel{\pi}{\simeq}\A_{0}} \arrow{e,t}{{\rm id}}
\node{\ker(\varphi_{T^{*}})\stackrel{\pi}{\simeq} \A}\arrow{e,t}{\epsilon_{\infty}}
\node{\K(\ell^{2}(\N))}\arrow{e} \node{0}
\end{diagram}
\end{equation}

Finally for the system $({\bf c}/{\bf c_{0}},\N,\tilde{\tau})$, we first note that it is equivariant to $(\C,\N,{\rm id})$.
So in this case, we have $r K(\ell^{2}(N,\C))r=K(\ell^{2}(\N))$, and
$\C\times_{\rm id}^{\piso}\N \stackrel{\rho}{\simeq} \T(\Z)$ where the isomorphism $\rho$ is given by the partial-isometric representation $n\mapsto T_{n}^{*}$, and
identify $(\C\times_{\rm id}^{\iso}\N, j_{\N}) \simeq \C\times_{\rm id}\Z\simeq (C^{*}(\Z),u)$ with the algebra $C(\TT)$ of continuous functions on $\TT$ using
$\delta: j_{\N}(n)\mapsto u_{-n}\in C^{*}(\Z) \mapsto (z \mapsto \overline{z}^{n}) \in C(\TT)$.
Then we get the third row exact sequence of diagram (6.1) of \cite[Theorem 6.1]{LR}:
\begin{equation}\label{diagram-CC}
\begin{diagram}\dgARROWLENGTH=0.3\dgARROWLENGTH
\node{0} \arrow{e} \node{K(\ell^{2}(\N))} \arrow{se}\arrow{e,t}{\Psi}
\node{\C \times_{\id}^{\piso}\N} \arrow{s,r}{\rho}\arrow{e,t}{\phi}
\node{\C \times_{\id}^{\iso}\N} \arrow{s,r}{\delta}\arrow{e} \node{0}
\\
\node{ } \node{ } \node{\T(\Z)} \arrow{e,t}{\psi_{T}}
\node{C(\TT)} \arrow{ne} \node{}
\end{diagram}
\end{equation}

\begin{remark}
We have seen in Example \ref{example-LR} the three rows exact sequences of \cite[Diagram 6.1]{LR} are computed from our results.
The three columns exact sequences can actually be obtained by \cite[Theorem 2.2,Corollary 3.1]{AH}.
Although these do not imply the commutativity of all rows and columns (because we have not obtained the analogous theorem of \cite[Theorem 2.2]{AH} for the algebra $\T_{(A,\N,\alpha)}$), nevertheless it follows from our results that
the algebras $\A$ and $\A_{0}$ appeared in \cite[Diagram 6.1]{LR}  are Morita equivalent to ${\bf c}\otimes K(\ell^{2}(\N))$ and ${\bf c_{0}}\otimes K(\ell^{2}(\N))$ respectively.
It is a helpful fact in particular for describing the primitive ideal space of ${\bf c}\times_{\tau}^{\piso}\N$.
\end{remark}

\begin{example}
If $(A,\N,\alpha)$ is a system of a $C^{*}$-algebra for which $\overline{\alpha}(1)=1$, then (\ref{diagram1}) is the exact sequence of \cite[Theorem 1.5]{KS}.
This is because  $p=\overline{\pi}_{\alpha}(1)$ is the identity of $\T_{(A,\N,\alpha)}$, so $A\times_{\alpha}^{\piso}\N$ is isomorphic to $\T_{(A,\N,\alpha)}$ and
$p[\K(\ell^{2}(\N,A))]p=\K(\ell^{2}(\N,A))$.
Let $(A_{\infty},\beta^{n})_{n}$ be the limit of direct sequence $(A_{n})$ where $A_{n}=A$ for every $n$ and $\alpha_{m-n}:A_{n}\rightarrow A_{m}$ for $n\le m$.
All the bonding maps $\beta^{i}:A_{i}\rightarrow A_{\infty}$  extend trivially to the multiplier algebras and preserve the identity.
Therefore we have $(A\times^{\iso}_{\alpha}\N,j_{A},j_{\N})\simeq (A_{\infty}\times_{\alpha_{\infty}}\Z,i_{\infty},u)$ in which the isomorphism is given by
$\iota(j_{\N}(n)^{*}j_{A}(a)j_{\N}(m)=u_{n}^{*}i_{\infty}(\beta^{0}(a))u_{m}$, and then the commutative diagram follows:
\begin{equation}\label{diagram-KS}
\begin{diagram}\dgARROWLENGTH=0.3\dgARROWLENGTH
\node{0} \arrow{e} \node{p~[K(\ell^{2}(\N, A))]~p} \arrow{s,r}{\id} \arrow{e,t}{\Psi}
\node{A\times_{\alpha}^{\piso}\N} \arrow{s,r}{\psi}\arrow{e,t}{\phi}
\node{A \times_{\alpha}^{\iso}\N} \arrow{s,r}{\iota} \arrow{e} \node{0}
\\
\node{0} \arrow{e} \node{K(\ell^{2}(\N, A))} \arrow{e,t}{{\rm id}}
\node{\T_{(A,\N,\alpha)}}\arrow{e,t}{q}
\node{A_{\alpha_{\infty}}\times_{\alpha_{\infty}}\Z}\arrow{e} \node{0.}
\end{diagram}
\end{equation}
\end{example}

\section{The partial-isometric crossed product of a system by a semigroup of automorphisms.}
Suppose $(A,\Gamma^{+},\alpha)$ is a system of an action $\alpha:\Gamma^{+} \rightarrow \Aut{A}$
by automorphisms on $A$,  and consider the distinguished system $(B_{\Gamma^{+}},\Gamma^{+},\tau)$ of the commutative $C^*$-algebra $B_{\Gamma^{+}}$
by semigroup of endomorphisms $\tau_{x}\in\End(B_{\Gamma^{+}})$.
Then $x\mapsto \tau_{x}\otimes \alpha_{x}^{-1}$ defines an action $\gamma$ of $\Gamma^{+}$ by endomorphisms of $B_{\Gamma^{+}}\otimes A$.
So we have a system $(B_{\Gamma^{+}}\otimes A,\Gamma^{+},\gamma)$ by a semigroup of endomorphisms.
We prove in the proposition below that the isometric-crossed product $(B_{\Gamma^{+}}\otimes A) \times_{\gamma}^{\iso} \Gamma^{+}$
is $A\times_{\alpha}^{\piso}\Gamma^{+}$.

\begin{prop}\label{piso-iso}
Suppose $\alpha:\Gamma^{+} \rightarrow \Aut{A}$ is an action by automorphisms on a $C^*$-algebra $A$
of the positive cone $\Gamma^{+}$ of a totally ordered abelian group $\Gamma$.
Then the partial-isometric crossed product $A\times_{\alpha}^{\piso}\Gamma^{+}$ is isomorphic to the isometric crossed product
$((B_{\Gamma^{+}}\otimes A) \times_{\gamma}^{\iso} \Gamma^{+},j)$.
More precisely, the $C^{*}$-algebra $(B_{\Gamma^{+}}\otimes A) \times_{\gamma}^{\iso} \Gamma^{+}$ together with a pair of homomorphisms
$(k_{A},k_{\Gamma^{+}}):(A,\Gamma^{+},\alpha)\rightarrow M((B_{\Gamma^{+}}\otimes A) \times_{\gamma}^{\iso} \Gamma^{+})$ defined by
$k_{A}(a)=j_{B_{\Gamma^{+}}\otimes A}(1\otimes a)$ and
$k_{\Gamma^{+}}(x)=j_{\Gamma^{+}}(x)^{*}$ is a partial-isometric crossed product for $(A,\Gamma^{+},\alpha)$.
\end{prop}

\begin{proof}
Every $k_{\Gamma^{+}}(x)$ satisfies
$k_{\Gamma^{+}}(x)k_{\Gamma^{+}}(x)^{*}=j_{\Gamma^{+}}(x)^{*}j_{\Gamma^{+}}(x)=1$, and
$(k_{A},k_{\Gamma^{+}})$ is a partial-isometric covariant representation for $(A,\Gamma^{+},\alpha)$:
\begin{align*}
j_{B_{\Gamma^{+}}\otimes A}(1\otimes \alpha_{x}(a))
& = j_{\Gamma^{+}}(x)^{*}j_{\Gamma^{+}}(x)j_{B_{\Gamma^{+}}\otimes A}(1\otimes \alpha_{x}(a))j_{\Gamma^{+}}(x)^{*}j_{\Gamma^{+}}(x) \\
& = j_{\Gamma^{+}}(x)^{*}j_{B_{\Gamma^{+}}\otimes A}(\tau_{x}\otimes \alpha_{x}^{-1}(1\otimes \alpha_{x}(a)))j_{\Gamma^{+}}(x)\\
& = j_{\Gamma^{+}}(x)^{*}j_{B_{\Gamma^{+}}\otimes A}(1_{x}\otimes a)j_{\Gamma^{+}}(x)
= j_{\Gamma^{+}}(x)^{*}j_{B_{\Gamma^{+}}}(1_{x})j_{A}(a)j_{\Gamma^{+}}(x)\\
& = j_{\Gamma^{+}}(x)^{*}j_{\Gamma^{+}}(x)j_{\Gamma^{+}}(x)^{*}j_{A}(a)j_{\Gamma^{+}}(x)
= j_{\Gamma^{+}}(x)^{*}j_{B_{\Gamma^{+}}\otimes A}(1\otimes a)j_{\Gamma^{+}}(x),
\end{align*}
and
$j_{\Gamma^{+}}(x)j_{\Gamma^{+}}(x)^{*}j_{B_{\Gamma^{+}}\otimes A}(1\otimes a)=
j_{B_{\Gamma^{+}}\otimes A}(1\otimes a)j_{\Gamma^{+}}(x)j_{\Gamma^{+}}(x)^{*}$ because
$j_{B_{\Gamma^{+}}\otimes A}(1_{x}\otimes a)=j_{A}(a)j_{B_{\Gamma^{+}}}(1_{x})$

Suppose $(\pi,V)$ is a partial-isometric covariant representation of $(A,\Gamma^{+},\alpha)$ on $H$.
We want to get a non degenerate representation $\pi\times V$ of the isometric crossed product
$(B_{\Gamma^{+}}\otimes A) \times_{\gamma}^{\iso} \Gamma^{+}$ which satisfies
$(\pi\times V)\circ k_{A}(a)=\pi(a)$ and $(\overline{\pi\times V})\circ k_{\Gamma^{+}}(x)=V_{x}$ for all $a\in A$ and $x\in\Gamma^{+}$.

Since $V_{x}V_{x}^{*}=1$ for all $x\in\Gamma^{+}$,
$x\mapsto V_{x}^{*}$ is an isometric representation of $\Gamma^{+}$, and therefore
$\pi_{V^{*}}(1_{x})=V_{x}^{*}V_{x}$ defines a representation $\pi_{V^{*}}$ of $B_{\Gamma^{+}}$
such that $(\pi_{V^{*}},V^{*})$ is an isometric covariant representation of $(B_{\Gamma^{+}},\Gamma^{+},\tau)$.
Moreover $\pi_{V^{*}}$ commutes with $\pi$ because
$\pi_{V^{*}}(1_{x})\pi(a)=V_{x}^{*}V_{x}\pi(a)=\pi(a)V_{x}^{*}V_{x}=\pi(a)\pi_{V^{*}}(1_{x})$.
Thus $\pi_{V^{*}}\otimes \pi$ is a non degenerate representation of $B_{\Gamma^{+}}\otimes A$ on $H$, and
$\pi_{V^{*}}\otimes \pi(1_{y}\otimes a)=\pi_{V^{*}}(1_{y})\pi(a)=\pi(a)\pi_{V^{*}}(1_{y})$.
We clarify that $(\pi_{V^{*}}\otimes \pi, V^{*})$ is in fact an isometric covariant representation of the system
$(B_{\Gamma^{+}}\otimes A,\Gamma^{+},\gamma)$:
\begin{align*}
\pi_{V^{*}}\otimes \pi(\tau_{x}\otimes \alpha_{x}^{-1}(1_{y}\otimes a)) & = \pi_{V^{*}}(\tau_{x}(1_{y}))\pi(\alpha_{x}^{-1}(a))
=V_{x}^{*}\pi_{V^{*}}(1_{y})V_{x}\pi(\alpha_{x}^{-1}(a)) \\
& =V_{x}^{*}\pi_{V^{*}}(1_{y})\pi(\alpha_{x}(\alpha_{x}^{-1}(a)))V_{x} \mbox{ by piso covariance of } (\pi,V) \\
& =V_{x}^{*}\pi_{V^{*}}(1_{y})\pi(a)V_{x} =V_{x}^{*}(\pi_{V^{*}}\otimes \pi)(1_{y}\otimes a)V_{x}.
\end{align*}
Then $\rho:=(\pi_{V^{*}}\otimes \pi)\times V^{*}$ is the non degenerate representation
of $(B_{\Gamma^{+}}\otimes A) \times_{\gamma}^{\iso}\Gamma^{+}$
which satisfies the requirements
\[ \rho(k_{A}(a))=\rho(j_{B_{\Gamma^{+}}\otimes A }(1\otimes a))=\pi_{V^{*}}\otimes\pi(1\otimes a)=\pi(a) \]
and $\overline{\rho}(k_{\Gamma^{+}}(x))=\overline{\rho}(j_{\Gamma^{+}}(x)^{*})=V_{x}$.
Finally, the span of $\{k_{\Gamma^{+}}(x)^{*}k_{A}(a)k_{\Gamma^{+}}(y)\}$ is dense in
$(B_{\Gamma^{+}}\otimes A) \times_{\gamma}^{\iso}\Gamma^{+}$ because
\[ k_{\Gamma^{+}}(x)^{*}k_{A}(a)k_{\Gamma^{+}}(y)=j_{\Gamma^{+}}(y)^{*}j_{B_{\Gamma^{+}}\otimes A }(1_{x+y}\otimes \alpha_{x+y}^{-1}(a))j_{\Gamma^{+}}(x). \]
\end{proof}

Proposition \ref{piso-iso} gives an isomorphism
$k: (A\times_{\alpha}^{\piso}\Gamma^{+},i)\rightarrow ((B_{\Gamma^{+}}\otimes A) \times_{\gamma}^{\iso} \Gamma^{+},j)$
which satisfies
$k(i_{\Gamma^{+}}(x))=j_{\Gamma^{+}}(x)^{*}$ and $k(i_{A}(a))=j_{B_{\Gamma^{+}}\otimes A}(1\otimes a)$.
This isomorphism maps the ideal $\ker\phi$ of $A\times_{\alpha}^{\piso}\Gamma^{+}$ in Proposition \ref{kernel-piso-iso} isomorphically onto
the ideal
\[ {\mathcal I} := \overline{\newspan}\{j_{B_{\Gamma^{+}\otimes A}}(1\otimes a) j_{\Gamma^{+}}(x) [1-j_{\Gamma^{+}}(t)j_{\Gamma^{+}}(t)^{*}] j_{\Gamma^{+}}(y)^{*} : a\in A, x,y,t \in\Gamma^{+}\} \]
of $(B_{\Gamma^{+}}\otimes A) \times_{\gamma}^{\iso} \Gamma^{+}$.
We identify this ideal in Lemma \ref{ker-iso}.
First we need to recall from \cite{Adji1} the notion of extendible ideals, it was shown there that
\[ B_{\Gamma^{+},\infty}:=\overline{\newspan}\{1_{x}-1_{y} : x<y \in\Gamma^{+}\} \]
is an extendible $\tau$-invariant ideal of $B_{\Gamma^{+}}$.
Thus $B_{\Gamma^{+},\infty}\otimes A$ is an extendible $\gamma$-invariant ideal of $B_{\Gamma^{+}}\otimes A$.
We can therefore consider the system $(B_{\Gamma^{+},\infty}\otimes A),\Gamma^{+},\gamma)$.
Extendibility of ideal is required to assure the crossed product
$(B_{\Gamma^{+},\infty}\otimes A)\times_{\gamma}^{\iso}\Gamma^{+}$ embeds naturally as an ideal of $(B_{\Gamma^{+}}\otimes A)\times_{\gamma}^{\iso}\Gamma^{+}$ such that the quotient is the crossed product of the quotient algebra $B_{\Gamma^{+}}\otimes A /B_{\Gamma^{+},\infty}\otimes A$ \cite[Theorem 3.1]{Adji1}.

\begin{lemma} \label{ker-iso}
The ideal ${\mathcal I}$ is $(B_{\Gamma^{+},\infty}\otimes A) \times_{\gamma}^{\iso}\Gamma^{+}$.
\end{lemma}

\begin{proof}
We know from \cite[Theorem 3.1]{Adji1} that the ideal $(B_{\Gamma^{+},\infty}\otimes A)\times_{\gamma}^{\iso}\Gamma^{+}$ is spanned by
\[ \{j_{\Gamma^{+}}(v)^{*} j_{B_{\Gamma^{+}\otimes A}}((1_{s}-1_{t})\otimes a)j_{\Gamma^{+}}(w):  s< t, v,w \text{ in } \Gamma^{+}, a\in A\}.\]
So to prove the Lemma, it is enough to show that ${\mathcal I}$ and $(B_{\Gamma^{+},\infty}\otimes A)\times_{\gamma}^{\iso}\Gamma^{+}$ contain each other.

We compute on their generator elements in next paragraph using the fact that the covariant representation $(j_{B_{\Gamma^{+}}\otimes A},j_{\Gamma^{+}})$ gives a unital homomorphism $j_{B_{\Gamma^{+}}}$ which commutes with the non degenerate homomorphism $j_{A}$, and that the pair
$(j_{B_{\Gamma^{+}}},j_{\Gamma^{+}})$ is a covariant representation of $(B_{\Gamma^{+}},\Gamma^{+},\tau)$.
Each isometry $j_{\Gamma^{+}}(x)$ is not a unitary, so the pair $(j_{A},j_{\Gamma^{+}})$ fails to be a covariant representation of $(A,\Gamma^{+},\alpha^{-1})$.
However it satisfies the equation $j_{A}(\alpha^{-1}_{x}(a))j_{\Gamma^{+}}(x)=j_{\Gamma^{+}}(x)j_{A}(a)$ for all $a\in A$ and $x\in\Gamma^{+}$.

Let $\xi$ be a spanning element of ${\mathcal I}$.
If $x<y$ and $t$ are in $\Gamma^{+}$, then $j_{\Gamma^{+}}(y)^{*}=j_{\Gamma^{+}}(x)^{*}j_{\Gamma^{+}}(y-x)^{*}$, and
\begin{align*}
j_{\Gamma^{+}}(x)[1-j_{\Gamma^{+}}(t)j_{\Gamma^{+}}(t)^{*}] j_{\Gamma^{+}}(y)^{*} & =
(j_{\Gamma^{+}}(x)j_{\Gamma^{+}}(x)^{*}-j_{\Gamma^{+}}(x+t)j_{\Gamma^{+}}(x+t)^{*})j_{\Gamma^{+}}(y-x)^{*}\\
& = \overline{j}_{B_{\Gamma^{+}}\otimes A}((1_{x}-1_{x+t})\otimes 1_{M(A)}) j_{\Gamma^{+}}(y-x)^{*},
\end{align*}
so
\begin{align*}
\xi=j_{B_{\Gamma^{+}}\otimes A}((1_{x}-1_{x+t})\otimes a)j_{\Gamma^{+}}(y-x)^{*}
& = j_{\Gamma^{+}}(y-x)^{*} j_{B_{\Gamma^{+}}\otimes A}(\gamma_{y-x}((1_{x}-1_{x+t})\otimes a)) \\
& = j_{\Gamma^{+}}(y-x)^{*} j_{B_{\Gamma^{+}}\otimes A}((1_{y}-1_{y+t})\otimes \alpha_{y-x}^{-1}(a)).
\end{align*}
If $x\ge y$, then $j_{\Gamma^{+}}(x) =j_{\Gamma^{+}}(x-y) j_{\Gamma^{+}}(y)$, and
\begin{align*}
j_{\Gamma^{+}}(x) & [1-j_{\Gamma^{+}}(t)j_{\Gamma^{+}}(t)^{*}] j_{\Gamma^{+}}(y)^{*} =
j_{\Gamma^{+}}(x-y)[j_{\Gamma^{+}}(y)j_{\Gamma^{+}}(y)^{*}-j_{\Gamma^{+}}(y+t)j_{\Gamma^{+}}(y+t)^{*}] \\
& = j_{\Gamma^{+}}(x-y)\overline{j}_{B_{\Gamma^{+}}\otimes A}((1_{y}-1_{y+t})\otimes 1_{M(A)})j_{\Gamma^{+}}(x-y)^{*}j_{\Gamma^{+}}(x-y) \\
& = \overline{j}_{B_{\Gamma^{+}}\otimes A}((1_{x}-1_{x+t})\otimes 1_{M(A)})j_{\Gamma^{+}}(x-y),
\end{align*}
so  $\xi=j_{B_{\Gamma^{+}}\otimes A}((1_{x}-1_{x+t})\otimes a)j_{\Gamma^{+}}(x-y)$, and therefore
${\mathcal I}\subset (B_{\Gamma^{+},\infty}\otimes A)\times_{\gamma}^{\iso}\Gamma^{+}$.

For the other inclusion, let $\eta=j_{B_{\Gamma^{+}}\otimes A} ((1_{s}-1_{t})\otimes a)j_{\Gamma^{+}}(x)$ be a generator of
$(B_{\Gamma^{+},\infty}\otimes A)\times_{\gamma}^{\iso}\Gamma^{+}$.
Then $\eta=j_{A}(a) [j_{\Gamma^{+}}(s)j_{\Gamma^{+}}(s)^{*}-j_{\Gamma^{+}}(t)j_{\Gamma^{+}}(t)^{*}]j_{\Gamma^{+}}(x)$,
and a similar computation shows that
\begin{align*} & [j_{\Gamma^{+}}(s)j_{\Gamma^{+}}(s)^{*} - j_{\Gamma^{+}}(t)j_{\Gamma^{+}}(t)^{*}]  j_{\Gamma^{+}}(x) \\
&=\left\{ \begin{array}{ll}
j_{\Gamma^{+}}(s)[1-j_{\Gamma^{+}}(t-s)j_{\Gamma^{+}}(t-s)^{*}]j_{\Gamma^{+}}(s-x)^{*} & \mbox{ for } x\le s< t \\
j_{\Gamma^{+}}(x)[1-j_{\Gamma^{+}}(t-x)j_{\Gamma^{+}}(t-x)^{*}] & \quad \quad s<x <t \\
0  & \mbox{ for } t=x \mbox{ or } s<t<x,
\end{array}
\right.
\end{align*}
which implies that $\eta\in {\mathcal I}$.
\end{proof}

\bigskip
An isometric crossed product is isomorphic to a full corner in the ordinary crossed product by the dilated action.
The action $\tau:\Gamma^{+}\rightarrow\End(B_{\Gamma^{+}})$ is dilated to the action $\tau:\Gamma\rightarrow \Aut(B_{\Gamma})$ where $\tau_{s}(1_{x})=1_{x+s}$ acts on the algebra $B_{\Gamma}=\overline{\newspan}\{1_{x} : x \in\Gamma\}$.
We refer to Lemma 3.2 of \cite{Adji2} to see that a dilation of $(B_{\Gamma^{+}}\otimes A,\Gamma^{+},\gamma)$
gives the system $(B_{\Gamma}\otimes A,\Gamma,\gamma_{\infty})$, in which
$\gamma_{\infty}=\tau\otimes\alpha^{-1}$ acts by automorphisms on the algebra $B_{\Gamma}\otimes A$.
The bonding homomorphism $h_{s}$ for $s\in\Gamma^{+}$, is given by
\[ h_{s}: (1_{x}\otimes a)\in B_{\Gamma^{+}}\otimes A ~ \mapsto ~ (1_{x}\otimes a)\in \overline{\newspan}\{1_{y}: y\ge -s\}\otimes A
\hookrightarrow B_{\Gamma}\otimes A. \]
This homomorphism extends to the multiplier algebras, we write as $\overline{h}_{0}$, and it carries the identity
$1_{0}\otimes 1_{M(A)} \in M(B_{\Gamma^{+}}\otimes A)$ into the projection
$\overline{h}_{0}(1_{0}\otimes 1_{M(A)}) \in M(B_{\Gamma}\otimes A)$.
Let
\[ p:=\overline{j}_{B_{\Gamma}\otimes A}(\overline{h}_{0}(1_{0}\otimes 1_{M(A)})) \]
be the projection in the crossed product $M((B_{\Gamma}\otimes A)\times_{\gamma_{\infty}}\Gamma)$.
Then it follows from \cite[Theorem 2.4]{Adji1} or \cite[Theorem 2.4]{Laca} that $(B_{\Gamma^{+}}\otimes A) \times_{\gamma}^{\iso} \Gamma^{+}$ is isomorphic onto
the full corner $p~[(B_{\Gamma}\otimes A)\times_{\gamma_{\infty}}\Gamma]~p$.

\begin{cor}\label{morita}
There is an isomorphism of $A\times_{\alpha}^{\piso}\Gamma^{+}$ onto the full corner
$p~[(B_{\Gamma}\otimes A)\times_{\gamma_{\infty}}\Gamma]~p$ of the crossed product $(B_{\Gamma}\otimes A)\times_{\gamma_{\infty}}\Gamma$, such that
the ideal $\ker\phi$ of $A\times_{\alpha}^{\piso}\Gamma^{+}$ in Proposition \ref{kernel-piso-iso} is isomorphic onto
the ideal $p~[(B_{\Gamma,\infty}\otimes A)\times_{\gamma_{\infty}}\Gamma]~p$, where
$B_{\Gamma,\infty}=\overline{\newspan}\{1_{s}-1_{t} : s<t \in\Gamma\}$.
\end{cor}

\begin{cor}\label{trivial-action}
Suppose $\alpha:\Gamma^{+}\rightarrow {\rm Aut}(A)$ is the trivial action $\alpha_{x}={\rm identity}$ for all $x$, and let
${\mathcal C}_{\Gamma}$ denote the commutator ideal of the Toeplitz algebra ${\mathcal T}(\Gamma)$.
Then there is a short exact sequence
\begin{equation}\label{diagram2}
\begin{diagram}\dgARROWLENGTH=0.5\dgARROWLENGTH
\node{0} \arrow{e}
\node{A\otimes {\mathcal C}_{\Gamma}} \arrow{e}
\node{A\times_{\alpha}^{\piso}\Gamma^{+}} \arrow{e,t}{\phi}
\node{A\times_{\alpha}\Gamma} \arrow{e}
\node{0.}
\end{diagram}
\end{equation}
\end{cor}

\begin{proof}
We have already identified in Lemma \ref{ker-iso} that the ideal $\I$ is $(B_{\Gamma^{+},\infty}\otimes A)\times_{\tau\otimes{\rm id}}^{\iso}\Gamma^{+}$.
We know that we have a version of \cite[Lemma 2.75]{Dana} for isometric crossed product, which says that if $(C,\Gamma^{+},\gamma)$ is a dynamical system and $D$ is any $C^*$-algebra, then $(C\otimes_{\max}D)\times_{\gamma\otimes{\rm id}}^{\iso}\Gamma^{+}$ is isomorphic to $(C\times_{\gamma}^{\iso}\Gamma^{+})\otimes_{\max} D$.
Applying this to the system $(B_{\Gamma^{+},\infty},\Gamma^{+},\tau)$ and the $C^*$-algebra $A$,  we get
\[ (B_{\Gamma^{+},\infty}\otimes A)\times_{\tau\otimes{\rm id}}^{\iso}\Gamma^{+}\simeq (B_{\Gamma^{+},\infty}\times_{\tau}^{\iso}\Gamma^{+})\otimes A\simeq {\mathcal C}_{\Gamma}\otimes A, \]
and hence we obtained the exact sequence.
\end{proof}

\begin{remark}
Note that
\[ A\times_{\rm id}^{\piso}\Gamma^{+}\simeq (B_{\Gamma^{+}}\otimes A)\times_{\tau\otimes{\rm id}}^{\iso}\Gamma^{+}\simeq (B_{\Gamma^{+}}\times_{\tau}^{\iso}\Gamma^{+})\otimes A\simeq \T(\Gamma)\otimes A, \]
and $A\times_{\rm id}^{\iso}\Gamma^{+}\simeq A\times_{\rm id}\Gamma \simeq A\otimes C^{*}(\Gamma)\simeq A\otimes C(\hat{\Gamma})$.
So (\ref{diagram2}) is the exact sequence
\[
\begin{diagram}\dgARROWLENGTH=0.5\dgARROWLENGTH
\node{0} \arrow{e}
\node{A\otimes {\mathcal C}_{\Gamma}} \arrow{e}
\node{A\otimes \T(\Gamma)} \arrow{e,t}{\phi}
\node{A\otimes C(\hat{\Gamma})}  \arrow{e}
\node{0,}
\end{diagram}
\]
which is the (maximal) tensor product with the algebra $A$ to the well-known exact sequence
$0\rightarrow {\mathcal C}_{\Gamma}\rightarrow \T(\Gamma)\rightarrow C(\hat{\Gamma})\rightarrow 0$.
\end{remark}

\subsection{The extension of Pimsner Voiculescu}
Consider a system $(A,\Gamma^{+},\alpha)$ in which every $\alpha_{x}$ is an automorphism of $A$.
Let $(A\times_{\alpha}\Gamma,j_{A},j_{\Gamma})$ be the corresponding group crossed product.
The Toeplitz algebra $\T(\Gamma)$  is the $C^{*}$-algebra  generated by semigroup $\{T_{x}: x\in\Gamma^{+}\}$ of non unitary isometries $T_{x}$, and the commutator ideal ${\mathcal C}_{\Gamma}$ of $\T(\Gamma)$ generated by the elements $T_{s}T_{s}^{*}-T_{t}T_{t}^{*}$ for $s<t$ is given by
$\overline{\newspan}\{T_{r}(1-T_{u}T_{u}^{*})T_{t}^{*}: r,u,t\in\Gamma^{+}\}$ of $\T(\Gamma)$.

Consider the $C^*$-subalgebra $\T_{PV}(\Gamma)$ of $M((A\times_{\alpha}\Gamma)\otimes \T(\Gamma))$ generated by $\{j_{A}(a)\otimes I: a\in A\}$ and
$\{j_{\Gamma}(x)\otimes T_{x} : x\in\Gamma^{+}\}$.
Let ${\mathcal S}(\Gamma)$ be the ideal of $\T_{PV}(\Gamma)$ generated by
$\{j_{A}(a)\otimes  (T_{s}T_{s}^{*}-T_{t}T_{t}^{*}): s<t \in\Gamma^{+},a\in A\}$.

We claim that $(A\times_{\alpha^{-1}}^{\piso}\Gamma^{+},i_{A},i_{\Gamma^{+}}) \simeq \T_{PV}(\Gamma)$, and the isomorphism takes the ideal $\ker(\phi)$ onto
${\mathcal S}(\Gamma)$.
To see this, let $\pi(a):=j_{A}(a)\otimes I$ and $V_{x}:=j_{\Gamma}(x)^{*}\otimes T_{x}^{*}$.
Then $(\pi,V)$ is a partial-isometric covariant representation of $(A,\Gamma^{+},\alpha^{-1})$ in the $C^*$-algebra
$M((A\times_{\alpha}\Gamma)\otimes\T(\Gamma))$.
So we have a homomorphism
$\psi:A\times_{\alpha^{-1}}^{\piso}\Gamma^{+}\rightarrow (A\times_{\alpha}\Gamma)\otimes\T(\Gamma)$ such that
\[
\psi( i_{A}(a)) = j_{A}(a)\otimes I \mbox{ and }
\overline{\psi}(i_{\Gamma^{+}}(x))=j_{\Gamma}(x)^{*}\otimes T_{x}^{*} \mbox{ for } a \in A, x\in\Gamma^{+}.
\]
Moreover for $a\in A$ and $x>0$, we have
\begin{eqnarray*}
\pi(a)(1-V_{x}^{*}V_{x}) & = & (j_{A}(a)\otimes I)(1-(j_{\Gamma}(x)\otimes T_{x})(j_{\Gamma}(x)^{*}\otimes T_{x}^{*})) \\
 & = & (j_{A}(a)\otimes I) - (j_{A}(a)\otimes I)(j_{\Gamma}(x)\otimes T_{x})(j_{\Gamma}(x)^{*}\otimes T_{x}^{*}) \\
 & = & (j_{A}(a)\otimes I) - (j_{A}(a)\otimes T_{x}T_{x}^{*}) \\
 & = & j_{A}(a)\otimes (I-T_{x}T_{x}^{*}).
\end{eqnarray*}
Since $T_{x}T_{x}^{*}\neq I$, the equation $\pi(a)(1-V_{x}^{*}V_{x})=0$ must imply $j_{A}(a)=0$ in $A\times_{\alpha}\Gamma$, and hence $a=0$ in $A$.
So by Theorem 4.8 \cite{LR}  the homomorphism $\psi$ is faithful.
Thus $A\times_{\alpha^{-1}}^{\piso}\Gamma^{+}\simeq \psi(A\times_{\alpha^{-1}}^{\piso}\Gamma^{+})=\T_{PV}(\Gamma)$.

The isomorphism $\psi:A\times_{\alpha^{-1}}^{\piso}\Gamma^{+}\rightarrow \T_{PV}(\Gamma)$ takes the ideal $\ker \phi$ of $A\times_{\alpha^{-1}}\Gamma^{+}$ to the algebra ${\mathcal S}(\Gamma)$.

\begin{cor}[The extension of Pimsner and Voiculescu] \hspace{3cm}\\
Let $(A,\N,\alpha)$ be a system in which $\alpha\in\Aut(A)$.
Then there is an exact sequence $0\rightarrow A\otimes \K(\ell^{2}(\N))\rightarrow \T_{PV}\rightarrow A\times_{\alpha}\Z\rightarrow 0$.
\end{cor}
\begin{proof}
Apply Theorem \ref{ext} to the system $(A,\N,\alpha^{-1})$, and then use the identifications $A\times_{\alpha^{-1}}^{\piso}\N\simeq \T_{PV}(\Z)$,
$\ker\phi\simeq {\mathcal S}(\Gamma)\simeq \K(\ell^{2}(\N,A))$ and $A\times_{\alpha} \Z\simeq A\times_{\alpha^{-1}} \Z$.
\end{proof}


\begin{thebibliography}{99}
\bibitem{Adji1} S. Adji, `Invariant ideals of crossed products by semigroups of endomorphisms',
{\it Proc. of conference in Functional Analysis and Global Analysis in Manila} (Oct 1996), Springer-Verlag Singapore, pp 1-8.

\bibitem{Adji2} S. Adji, `A remark on semigroup crossed products',
{\it Vietnam Journal of Math.}  31:4 (2003), 429-435.

\bibitem{Adji} S. Adji, `Semigroup crossed products and the structure of Toeplitz algebras',
{\it J.Operator Thoery} 44 (2000), 139-150.

\bibitem{ALNR} S. Adji, M. Laca, M. Nilsen and I. Raeburn, `Crossed products by semigroups of endomorphisms and the Toeplitz algebras of ordered groups',
{\it Proc. Amer. Math. Soc.} Vol.122, Number 4 (1994), 1133-1141.

\bibitem{AH} S. Adji and A. Hosseini, `The partial-isometric crossed products of $c_{0}$ by the forward and backward shifts',
{\it Bull. Malays. Math. Sci. Soc.} (2), 33(3), 2010, 487-498.

\bibitem{F} N. J. Fowler, `Discrete product systems of Hilbert bimodules', {\it Pacific J. Math.} 204 (2002), 335-375.

\bibitem{KS} M. Khoshkam and G. Skandalis, `Toeplitz algebras associated with endomorphisms and Pimsner-Voiculescu
exact sequences', {\it Pacific Journal of Mathematics} Vol 181 No 2 1997, 315-331.

\bibitem{Laca} M. Laca, `From endomorphisms to automorphisms and back: dilations and full corners',
{\it J. London Math. Soc.} (2) 61 (2000), 893-904.

\bibitem{Larsen} N. S. Larsen, `Nonunital semigroup crossed products', 
{\it Math. Proc. Royal Irish Acad.} 100A (2000), 205-218.

\bibitem{LR} J. Lindiarni and I. Raeburn, `Partial-isometric crossed products by semigroups of endomorphisms', 
{\it J. Operator Theory} 52 (2004), 61-87.

\bibitem{murphy} G.J. Murphy, `Ordered groups and Toeplitz algebras`, {\it J. Operator Theory} 18 (1987), 303-326.

\bibitem{murphy2} G.J. Murphy, `Ordered groups and crossed products of $C^*$-algebras', {\it Pacific J. Math.} 148 (1991), 319-349.

\bibitem{murphy3} G.J. Murphy, `Crossed products of $C^*$-algebras by semigroups of automorphisms', {\it Proc. London Math. Soc.} (3) 68  (1994), 423-448.

\bibitem{PV} M. Pimsner and D. Voiculescu, `Exact sequences for $K$-groups and $Ext$-groups of certain cross-product $C^*$-algebras', 
{\it J. Operator Theory} 4 (1980) No. 1, 93-118.

\bibitem{RW} I. Raeburn and D. P. Williams, {\it Morita equivalence and continuous-trace $C^*$-algebras}, AMS Mathematical Surveys and Monographs No.60, 1998.

\bibitem{stacey} P.J. Stacey, `Crossed products of $C^{*}$-algebras by $*$-endomorphisms',
{\it J. Austral. Math. Soc.} (Series A) 54 (1993), 204-212.

\bibitem{Dana} D. P. Williams, \emph{Crossed Products of $C^*$-Algebras}, AMS Mathematical Surveys and Monographs No.134, 2007.

\end{thebibliography}
\end{document}